\documentclass[12pt]{article}
\usepackage{epsfig,psfrag,amsmath,amssymb,latexsym}
\usepackage{amscd }
\usepackage[usenames,dvipsnames]{xcolor}
\usepackage{amsfonts}
\usepackage{graphicx}
\newcommand{\remove}[1]{}
\newtheorem{theorem}{Theorem}[section]

\newtheorem{example}{Example}[section]

\newtheorem{lemma}{Lemma}[section]

\newenvironment{proof}[1][Proof]{\textbf{#1.} }{\ \rule{0.5em}{0.5em}}
\newcommand{\N}{\mathbb{N}}

\newcommand{\Z}{\mathbb{Z}}
\numberwithin{equation}{section}

\begin{document}
\title{Reconstruction of a multidimensional scenery 
with a branching random walk} 
\author{ Heinrich Matzinger, \thanks{School of Mathematics, Georgia
 Institute of Technology, Atlanta, GA 30332.\hfill\break Email:
 matzi@math.gatech.edu}\;
Serguei Popov\thanks{Instituto
de Matem{\'a}tica, Estat{\'\i}stica e Computa\c{c}\~ao
Cient\'\i fica, rua S\'ergio Buarque de Holanda 651, Cidade
Universit\'aria, CEP 13083--859, Campinas SP, Brasil.\hfill\break
Email: popov@ime.unicamp.br}\; and  
Angelica Pachon\thanks{Department of  Mathematics, University of Turin, Via Carlo Alberto 10, 10123, Turin, Italy. 
\hfill\break Email: apachonp@unito.it}}

\maketitle

\begin{abstract} 
 In this paper we consider a $d$-dimensional
scenery seen along 
a simple symmetric branching random walk, where at each time each
particle  gives 
the color record it is seeing.   We show that
we can $a.s.$  reconstruct the scenery up to equivalence from the color
record of all the particles. 
For this we assume 
that the scenery has at least  $2d+1$ colors  which are $i.i.d.$ 
with uniform probability.  This is an improvement
in comparison to \cite{Popov-Angelica} where the particles needed
to see at 
each time
a window around their current position. 
In~\cite{Lowe-Matzinger-two-dim}
 the  reconstruction
is done for $d=2$ with only one particle instead of a branching
random walk,
but millions of colors are necessary.  
\end{abstract}


\section{Introduction}
The classical \textit{scenery reconstruction
problem} con\-siders a coloring $\xi:\mathbb{Z}\rightarrow\{1,2,\ldots,\kappa\}$ of the integers $\mathbb{Z}$.   To $\xi$ one refers as {\it the scenery} and to $\kappa$
as the {\it number of colors of  $\xi$}.  Furthermore,  a particle moves according to a  
a recurrent random walk $S$ on $\mathbb{Z}$ and at each instant of
time $t$ observes the color
$\chi_t:=\xi(S_t)$ at its current position $S_t$. The scenery reconstruction problem is formulated through the  following question: 
From a color record 
$\chi=\chi_0\chi_1\chi_2\ldots$, can one
 reconstruct the scenery $\xi$? The complexity of the scenery reconstruction problem varies with the number of colors $\kappa$, it generally increases as $\kappa$ decreases.

In \cite{Lindenstrauss} it was shown that  there are some strange
sceneries which can not be reconstructed. However, it is
possible  to show  that $a.s.$ a ``typical'' scenery, drawn at random according to a given
distribution, can be
reconstructed (possibly 
up to shift and/or reflection). 


In \cite{Heini2color}, \cite{Heini3color}, and  \cite{HeinisDiss} 
it is proved that almost every scenery can be
reconstructed in the one-dimensional case. 
In this instance  combinatorial methods are used, see for example 
\cite{messagetext}, \cite{Lowe-Matzinger-Merkl2001}, \cite{Matzinger-Angelica}, \cite{popovscen}, 
\cite{Matzinger-Rolles2001}, \cite{Matzinger-Rolles-small} and \cite{Matzinger-Rolles-polynomial}. However, all the techniques used in one
dimension completely fail  in two dimensions, see \cite{KestenReview}.

In \cite{Lowe-Matzinger-two-dim} a reconstruction algorithm 
for the $2$-dimensional case is provided. However,  in  \cite{Lowe-Matzinger-two-dim} 
the number of colors in the scenery  is very high (order of several billions). 
The $d$-dimensional case is approached in 
\cite{Popov-Angelica} using a branching random walk and assuming
that at each instant of time $t$, each particle is seeing the
observations contained in a box centered at its current location.

In the present article we show that we can reconstruct a $d$-dimensional 
scenery using again a branching random walk but only observing at each instant of
time $t$, the color at the current positions of the particles.  For this we
only need  at least $2d+1$ colors.
Our method exploits the combinatorial nature of the $d$-dimensional problem in an
entirely novel way.


The scenery reconstruction problem goes back
to questions from Kolmogorov, 
Kesten, Keane, Benjamini, Perez, Den Hollander and others.
A related well known problem is to {\it distinguish sceneries}: Benjamini, den
Hollander, and Keane independently asked whether all non-equivalent
sceneries could be distinguished. An outline of this problem is as follows: Let 
$\eta _{1}$ and $\eta _{2}$ be two given sceneries, and sssume that either $\eta
_{1}$ or $\eta _{2}$ is observed along a random walk path, but we do not
know which one. Can we figure out which of the two sceneries was taken?
Kesten and Benjamini in \cite{BenjaminiKesten} proved 
that one can distinguish almost  every pair of sceneries even in two dimensions and with only two colors. For this they take $\eta_{1}$ and $\eta_{2}$ both $i.i.d.$ and
independent of each other. 
Prior to that, Howard had shown in \cite{Howard1}, \cite{Howard2}, and \cite{Howard3} that
any two periodic one dimensional non-equivalent sceneries are
distinguishable, and that one can $a.s.$ distinguish single defects
in periodic sceneries.  {\it Detecting a single defect in a
scenery} refers to the problem of distinguishing two sceneries which
differ only in one point. Kesten proved in \cite{KestenSingDef} that one can $a.s.$
recognize a single defect in a random scenery with at least five colors.
As fewer colors lead to a reduced amount of information, it was conjectured that
it might not be possible to detect a single defect in a 2-color random scenery. However, in \cite{HeinisDiss} the first named author showed that the opposite 
is true - a single defect in a random 2-color scenery can be detected. He also proved that the whole scenery can be reconstructed without any prior knowledge. 

One motivation to study scenery reconstruction and distinguishing problems
was the {\it $T,T^{-1}$-problem} that origins in a famous conjecture in ergodic theory due to Kolmogorov. He demonstrated that every Bernoulli shift $T$ has a trivial tail-field and conjectured that also the converse is true. Let $\mathcal{K}$ denote the class of all transformations having a trivial tail-field. Kolmogorov's conjecture was shown to be wrong by Ornstein in \cite{Ornstein}, who presented an example of a transformation which is in $\mathcal{K}$ but not Bernoulli.  His transformation was particularly constructed to resolve Kolmogorov's 
conjecture.   In 1971 Ornstein,
Adler, and Weiss came up with a very natural example which is $\mathcal{K}$
but appeared not to be Bernoulli, see \cite{Weiss}.  This was the so called {\it $T,T^{-1}$-transformation}, and the $T,T^{-1}$-problem was to verify that it was not Bernoulli. It was solved by Kalikow in \cite{Kalikow}, showing that the $T,T^{-1}$-transformation is not even loosely Bernoulli.  A  generalization of this result was recently proved by den Hollander and Steif
\cite{denHollanderSteif}.\\

\section{Model and Statement of Results}
We consider a random coloring of the integers in $d$-dimension.
Let $\xi_z$ be $i.i.d.$ random variables, where the index
$z$ ranges over $\mathbb{Z}^d$. The variables $\xi_z$
take values from the set $\{0,1,2,\ldots,\kappa-1 \}$ where $\kappa
\geq4$, and  all the values from $\{0,1,2,\ldots,\kappa-1 \}$ have the same
probability. A realization of  $\xi=\{\xi_z\}_{z\in\Z^d}$  thus is a 
(random) coloring of $\mathbb{Z}^d$. We call this random field~$\xi$ a {\it $d$-dimensional scenery}. 

Now assume that a branching random walk (BRW) on $\mathbb{Z}^d$ observes
the $d$-dimensional scenery $\xi$, i.e.,  i.e. that each
particle of the BRW  observes the color at its current position. 

Formally a branching random walk in $\Z^d$ is described as
follows.  The process starts with  one particle at the origin, then at each
step, any particle is substituted by two particles with probability~$b$
and is left intact with probability~$1-b$, for some fixed
$b\in(0,1)$.  We denote by $N_n$ the total number of particles at
time~$n$. Clearly $(N_n, n=0,1,2,\ldots)$ is a Galton-Watson process with the 
branching probabilities ${\tilde p}_1=1-b$, ${\tilde p}_2=b$.
This process is supercritical, so it is clear that $N_n\to\infty$
$a.s.$  Furthermore, each particle follows the path of a simple random walk,
i.e., each particle jumps to one of its nearest neighbors chosen with
equal probabilities,  independently of everything else.

To give the formal definition of the observed process we first introduce some notation.
 Let $\eta_t(z)$ be the number of particles at site~$z \in \Z^d$ at time~$t\geq0$, with $\eta_0(z)=\text{\bf{1}}\{z=0\}$.  
We denote by $\eta_t=(\eta_t(z))_{z \in\Z^d}$ the configuration
at time~$t$ of the branching random walk on $\Z^d$, starting at the origin  with branching probability~$b$. 

Let $G$ be the genealogical tree of the Galton-Watson process, 
where $G_t=\{v_1^{t},\dots,v_{N_{t}}^{t}\}$  are the particles of $t$-th generation. Let  $S(v_j^t)$ be the position of $v_j^t$ in $\Z^d$, i.e., $S(v_j^t)=z$ if the $j$-th particle at time $t$ is located
on the site $z$. Recall that  we do not know the position of the
particles, only the color record made by the particles at every time, as well
as the  number of particles at each time.

According to our notations,  $\{ z \in\Z^d: \eta_t(z)\geq 1\}=\{\exists j;  S(v_j^t)=z, j=1,\dots,N_t\}$.
The {\it observations} $\chi$, to which we also refer as
{\it observed process}, is a coloring of the random tree $G$.
Hence the observations $\chi$  constitute a random map:
\begin{eqnarray*}
\chi: G_t&\rightarrow& \{0,1,2,\dots,\kappa-1\}\\
 v_i^t&\mapsto& \chi(v_i^t)=\xi(S(v_i^t)), \quad i=\{1,\dots,N_t\}.
\end{eqnarray*}
In other words, the coloring $\chi$ of the random tree $G$
yields the color the particle $i$ at time $t$ sees in the scenery from
where it is located at time $t$.\\

 Denote by $\Omega_1=\{(\eta_t)_{t\in  \N}\}$  the space of all
possible  ``evolutions'' of the branching random walk, 
by $\Omega_2=\{0,1,2,\ldots,\kappa-1\}^{\Z^d}$  the space of all possibles sceneries
$\xi$  and by  $\Omega_3=\{(\chi_t)_{t\in\N}\}$ the space of all 
possible realizations of the observed process.
We assume that $(\eta_t)_{t \in \N}$ and~$\xi$ are independent and 
distributed with the laws $P_1$ and $P_2$, respectively. 

Two sceneries~$\xi$ and~$\xi'$ are said to be  {\it equivalent}  (in this
case we write $\xi\thicksim\xi'$), if there exists an isometry
$\varphi:\Z^d\mapsto \Z^d$ such that $\xi(\varphi x)=\xi'(x)$ for all~$x$.

Now we  formulate the main result of this paper. The measure 
$P$ designates the product measure $P_1 \otimes P_2$.
\begin{theorem}
\label{teorema1}
For any $b\in (0,1)$ and $\kappa \geq 2d+1$ 
there exists a measurable function $\Lambda:\Omega_3 \to \Omega_2$ such that
$P(\Lambda(\chi) \thicksim \xi)=1$.  
\end{theorem}
The function $\Lambda$ represents an ``algorithm'',  the observations $\chi$ being its input and the reconstructed scenery $\xi$ (up to equivalence) being its output.
The main idea used to prove Theorem~\ref{teorema1} is to show how to
reconstruct a finite piece of the scenery (close to  the origin). 

\subsection{Main Ideas}
We start by defining a reconstruction algorithm with parameter~$n$
denoted by $\Lambda^n$, which works with all the observations
up to time $n^2$ to reconstruct a portion of  the scenery $\xi$
close to the origin. The portion reconstructed
should be the restriction of $\xi$ to a box with center
close to the origin. This means closer
than $\sqrt{n}$, which is a lesser order than the size of the
piece reconstructed. We will show that the algorithm
$\Lambda^n$ works with \textit{high probability}  ($whp$ for short, meaning with probability tending to one as $n\rightarrow\infty$).

Let us denote by $\mathcal{K}_x(s)$  the box of size $s$ 
centered at $x$ in $\Z^d$, i.e., $x+[-s,s]^d$, 
and  by $\mathcal{K}(s)$   the box of size $s$ 
centered at the origin of $\Z^d$. For a subset $A$ of $\mathbb{Z}^d$, we designate by
$\xi_A$ the restriction of $\xi$ to $A$, so $\xi_{\mathcal{K}_x(s)}$ denotes the restriction of $\xi$ to $\mathcal{K}_x(s)$.

In what follows, we will say that $w$ is a word of size $k$
of $\xi_A$, if it
can be read in a straight manner in $\xi_A$.
This means that $w$ is a word of $\xi_A$
if there exist an $x\in \mathbb{Z}^d$ and a canonical vector $\vec{e}$ (it defined to be one that
has only one non-zero entry equal to $+1$ or 
$-1$.),
so that $x+i\vec{e}\in A$, for all $i=0,1,2,\ldots,k-1$,  and
$$
w=\xi_{x}\xi_{x+\vec{e}}\xi_{x+2\vec{e}}\ldots
\xi_{x+(k-1)\vec{e}}.
$$

Let $A$ and $B$ be two subsets of $\mathbb{Z}^d$.
Then we say that   $\xi_A$ and $\xi_B$
 are equivalent to each other and write
$\xi_A\thicksim\xi_B$, if there exists an isometry $\varphi:A\mapsto B$ such that,
$\xi_A\circ \varphi=\xi_B$.

\subsection{The algorithm $\varLambda^n$ for reconstructing a finite piece of scenery close to the origin.} 
The four phases of this algorithm are described
in the following way:
\begin{enumerate}
\item{}{\bf First phase: Construction of short words of size $(\ln n)^2$.} The first phase aims at reconstructing words of  size $(\ln n)^2$ of $\xi_{\mathcal{K}(n^2)}$. The set of words
constructed in this phase is denoted by  $SHORTWORDS^n$. It
should hopefully  contain all words of size $(\ln n)^2$ in
$\xi_{\mathcal{K}(4n)}$, and be contained in the set of all words of size $(\ln n)^2$ in
$\xi_{\mathcal{K}(n^2)}$. The accurate definition of
$SHORTWORDS^n$  is as follows:  a word $w_2$ of size $(\ln n)^2$ is going to be selected to
be  in $SHORTWORDS^n$ if there exist two strings
$w_1$ and $w_3$ both of size $(\ln n)^2$ and such that:
\begin{enumerate}
\item{} $w_1w_2w_3$ appears in the observations before time
$n^2$, and
\item{} the only word $w$ of size $(\ln n)^2$ such that
$w_1ww_3$ appears in the observations up to time $n^4$ is $w_2$.
\end{enumerate}
Formally, let $W(\xi_{\mathcal{K}(4n)})$ and $W(\xi_{\mathcal{K}(n^2)})$ be the sets of all  words of size $(\ln n)^2$ in $\xi_{\mathcal{K}(4n)}$ and $\xi_{\mathcal{K}(n^2)}$
respectively, then 
\begin{eqnarray}\label{ShW}
W(\xi_{\mathcal{K}(4n)})\subseteq SHORTWORDS^n\subseteq
W(\xi_{\mathcal{K}(n^2)}).
\end{eqnarray}
We prove that (\ref{ShW}) holds $whp$ in Section~\ref{s_1phase}.

\item{}{\bf Second phase: Construction of long words of size $4n$.} The second phase assembles the words of $SHORTWORDS^n$
into longer words to construct another set of words denoted by
$LONGWORDS^n$. 
The rule is that the words of $SHORTWORDS^n$
to get assembled must coincide on $(\ln n)^2-1$ consecutive \-letters,  and it is done until getting strings of total 
size exactly equal to $4n+1$.
In this phase,  let $\mathcal{W}_{4n}(\xi_{\mathcal{K}(4n)})$ 
and $\mathcal{W}_{4n}(\xi_{\mathcal{K}(n^2)})$ be the set of all  words of size
$4n$ in $\xi_{\mathcal{K}(4n)}$ and $\xi_{\mathcal{K}(n^2)}$
respectively, then
\begin{eqnarray}\label{LW}
\mathcal{W}_{4n}(\xi_{\mathcal{K}(4n)})\subseteq LONGWORDS^n\subseteq 
\mathcal{W}_{4n}(\xi_{\mathcal{K}(n^2)}).
\end{eqnarray}
We achieve that (\ref{LW}) holds $whp$ in  Section~\ref{s_2phase}.

\item{}{\bf Third phase: Selecting a seed word close to the origin.} The third phase selects
 from the previous long words one which is close  to the
origin. For that $\varLambda^n$ applies the previous two phases, but with the parameter being equal to $n^{0.25}$ instead
of $n$. In other words, $\varLambda^n$ choses one (any) word $w_0$ of 
$LONGWORDS^{n^{0.25}}$, and then choses the word $w_L$ in
$LONGWORDS^n$ which contains $w_0$ in such a way that the relative position of $w_0$ 
inside  $w_L$ is centered. (The middle letters of $w_0$ and $w_L$ must coincide).
$\varLambda^n$ places the word $w_L$ so that the middle letter of
 $w_0$ is at the origin. See Section~\ref{s_3phase}.
  
\item{}{\bf Forth phase: Determining which long words are neighbors of each other.}
The fourth phase place  words from $LONGWORDS^n$
in the correct relative position to each other,  thus assembling the scenery in a box
near the origin.  For this, $\varLambda^n$ starts with the first long-word
which was placed close to the origin in the previous phase.
Then, words from the set $LONGWORDS^n$ are placed parallel to each
other until a piece of scenery on a box of size $4n$
is completed. 

Let us briefly explain  how $\varLambda^n$ choses which words of $LONGWORDS^n$ are {\it
neighbors} of each other in the scenery $\xi_{\mathcal{K}(n^2)}$, (i.e.,
they are parallel and at distance 1). 
$\varLambda^n$ estimates that the words $v$ and $w$
which appear in $\xi_{\mathcal{K}(n^2)}$ are neighbors of each
other iff the three following conditions are all satisfied:
\begin{enumerate}
\item First, there exist $4$ words $v_a$, $v_b$, $v_c$ and
$w_b$ having all size $(\ln n)^2$ except for $v_b$ which has
size $(\ln n)^2-2$, and such that
the concatenation $v_av_bv_c$ is contained in $v$, whilst up to
time $n^4$ it is observed at least once $v_aw_bv_c$.
\item Second,
the word $w_b$ is contained in $w$.
\item Finally, the relative position of 
$v_b$ in $v$ should be the same as the relative 
position of $w_b$ in $w$. By this we mean that the middle letter
of $v_b$ has the same position in $v$ as  the middle letter  
of $w_b$ in $w$ has.
\end{enumerate}
See the precise definition in Section~\ref{s_4phase}.
\end{enumerate}
Let $\mathcal{B}$ be the set of all finite binary trees, and let $\chi_t$ designates all the observations made by the branching random walk up to time $t$. That  $\chi_t$ is the restriction of the coloring $\chi$ to the sub-tree $\cup_{i\leq t} G_i$.

The next result means that the algorithm $\Lambda^n$
works $whp$.

\begin{theorem}
\label{maintheorem}Assume that the number of colors $\kappa$
satisfies the condition that $\kappa\geq 2d+1$. 
Then,
\label{n-theorem}
for every $n\in \mathbb{N}$ large enough, the  map
$$
\Lambda^n: \{0,1,\ldots,\kappa-1\}^\mathcal{B}\rightarrow \{0,1,\ldots,\kappa-1\}^{\mathcal{K}(4n)} ,
$$ defined above as our algorithm satisfies

\begin{equation}
P\left[\exists x\in \mathcal{K}(\sqrt{n})\text{ so that }\;\Lambda^n(\chi_{n^4})\sim \xi_{\mathcal{K}_x(4n)}\right]
\;\;\geq\;\;
1-\exp(-C(\ln n)^2),
\end{equation} where
$C>0$ is a constant  independent of $n$.
\end{theorem}

In  words, the algorithm $\Lambda^n$ manages to reconstruct $whp$ a piece of the scenery $\xi$ restricted to a box of size $4n$ close to the origin. The center of the box
has every coordinate not further than $\sqrt{n}$ from the origin.
The reconstruction algorithm uses only observations up to time $n^4$.

We will give the exact proof of the above theorem in the next section,
but, before we present the main ideas in a less formal way
in the remainder of this section. We first want to note
that the algorithm $\Lambda^n$ reconstructs a piece of the scenery $\xi$
in a box of size $4n$, but the exact position of that
piece within $\xi$ will in general not be known
after the reconstruction. Instead, the above theorem insures
that $whp$ the center of the box is not further
than an order $\sqrt{n}$ from the origin.

For what follows we will need  a few definitions:  Let 
$[0,k-1]$ designate the sequence
$\{0,1,2,3\ldots,k-1\}$ and $R$  be a map
$R:[0,k-1]\rightarrow\mathbb{Z}^d$ such that
the distance between $R(i)$ and $R(i+1)$ is $1$ for every
$i=0,1,2,\ldots,k-2$. We call such a map a {\it nearest neighbor
path
of length $k-1$}. Let $x$ and $y$ be two points in $\mathbb{Z}^d$.
If $R(0)=x$ and $R(k-1)=y$ we say that $R$ goes from $x$ to~$y$.
We also say that $R$ starts in $x$ and ends in $y$. Let $w$ be a 
string of colors
of size $k$, such that
$$w=\xi(R(0))\xi(R(1))\ldots\xi(R(k-1)).$$
In that case, we say that {\it  $R$ generates $w$ on the scenery $\xi$}.

\subsubsection{The DNA-sequencing trick}
We use the same trick as is used in modern methods
for DNA-sequencing where  instead of reconstructing the whole
DNA-sequence at once, one tries to reconstruct smaller pieces
simultaneously. Then to obtain the entire piece one 
puzzles the smaller pieces together. In this paper
the scenery is multidimensional unlike DNA-sequences.
Nonetheless, we first present the reconstruction
method for DNA-sequencing in this subsection, because 
it is easiest to understand. The trick goes as follows. 
Assume that you have a genetic sequence  $D=D_1D_2\ldots D_{n^\alpha}$,
where $\alpha>0$ is a constant independent of $n$,
and  $D$ is written in an alphabet  with $\kappa>0$ equiprobable letters. For instance  
in  DNA-sequences this alphabet is
$\{A,C,T,G\}$.  To determine the physical order of these letters in a sequence of DNA, modern methods  use the idea of do not go for the whole sequence at once, but first determine small  pieces (subsequences)  of order at least $C\ln(n)$ and then assemble them to obtain the whole sequence. (Here $C>0$ is a constant independent of $n$, but
dependent of $\alpha$).  Let us give an example.

 \begin{example}
Let the sequence be equal to
$$D=TATCAGT,$$
and suppose a biologist in his laboratory is able  to find all subsequences
of size $5$. Typically, the direction in which these subsequences
appear in $D$ is not known.  The set of all subsequences  of size $5$ which appear in $D$ is 
$$TATCA,ATCAG,TCAGT,$$
as well as their reverses
$$ACTAT,GACTA,TGACT.$$
Which one are to be read forward and which one backward is not
known to the biologist.  He is only given all these subsequences in one 
bag without extra information. If he was given all
the subsequences of size $5$ appearing in $D$, he could reconstruct
$D$ by assembling these subsequences using the assembly 
rule that they must coincide on a piece of size $4$. Why? Consider the  set
of all subsequences of size $4$:
$$TATC,ATCA,TCAG,CAGT$$
and their reverses
$$CTAT,ACTA,GACT,TGAC.$$
Note that each of these appears only once. Hence, four consecutive letters
determine uniquely the relative position of the subsequences of size $5$
with respect to each other, and  given the bag of subsequences
of size $5$ is possible to assemble theme one after the other. How?
Start by  picking any subsequence. Then  put down the next subsequence from 
the bag which  coincides with the previous one on at least $4$ letters.
For example take $GACTA$ and put it down on the integers in any position, for instance
$$
\begin{array}{c|c|c|c|c}
G&A&C&T&A\\\hline
0&1&2&3&4.
\end{array}
$$
Then, take another subsequence from the bag which coincides with the previously chosen one,
on at least $4$ contiguous letters. For example, $TGACT$ satisfies this condition. Now superpose the new subsequence onto the previous one  so that on $4$ letters they coincide:
$$
\begin{array}{c|c|c|c|c|c}
 T&G&A&C&T& \\
  &G&A&C&T&A\\\hline
 -1&0&1&2&3&4
\end{array}\;\;\;\text{.This leads to:}\;\;\;
\begin{array}{c|c|c|c|c|c}
 T &G&A&C&T&A\\\hline
  -1&0&1&2&3&4
\end{array}
$$
Next, observe that the subsequence $ACTAT$ coincides on at least $4$ consecutive letters, with what has been reconstructed so far. So,  put $ACTAT$ down in a matching position:
$$
\begin{array}{c|c|c|c|c|c|c}
   & &A&C&T&A&T\\  
 T &G&A&C&T&A&\\\hline
-1 &0&1&2&3&4&5
\end{array}
\;\;\;\text{.This leads to:}\;\;\;
\begin{array}{c|c|c|c|c|c|c}  
 T &G&A&C&T&A&T\\\hline
-1 &0&1&2&3&4&5
\end{array}
$$
The final result  is the sequence
$TGACTAT$ which is $D$ read in reverse order.
\end{example}

But how do we know that this method works?  in the example 
we saw the sequence $D$ before hand, and could hence verify that
each subsequence of size $4$ appears in at most one position
in $D$. However, the biologist can not verify this condition.
He only gets the bag of subsequences as only information. 
So, the idea is that if we know the stochastic model
which generates the sequence, we can calculate that $whp$ each subsequence of size $C\ln(n)$ appears only once in $D$, provided the constant $C>0$ is taken large enough.
Take for example the $i.i.d.$ model with $\kappa>1$ equiprobability
letters. Then the probability that two subsequences located in different and
non-intersecting parts of $D$ be identical is given by:
\begin{equation}
\label{star}
P(E^{n+}_{i,j})=\left(\frac{1}{\kappa}\right)^{C\ln n}=n^{-C\ln(\kappa)},
\end{equation}
where $E^{n+}_{i,j}$ is the event
that
\begin{equation}
\label{wordswords}
D_{i+1}\ldots D_{i+C\ln n}=D_{j+1}\ldots D_{j+C\ln n}.
\end{equation}
Similarly, let $E^{n-}_{i,j}$ be the event that
\begin{equation}
\label{wordswords1}
D_{i+1}\ldots D_{i+C\ln n}=D_{j-1}\ldots D_{j-C\ln n}.
\end{equation}

Note that even if the subsequences given on both  sides of (\ref{wordswords}) or (\ref{wordswords1})  intersect, as long as they are not exactly in the same position, we get that (\ref{star})  still
holds. To see this take for example the subsequence $w=D_{i+1}D_{i+2}\ldots D_{i+C(\ln n)}$
and the subsequence $v=D_{i+2}D_{i+3}\ldots D_{i+C(\ln n)+1}$. These two subsequences
are not at all independent of each other since up to two letters they are 
identical to each other. However we still have 
$$P(w=v)=P(D_{i+1}D_{i+2}\ldots D_{i+C(\ln n)}=
D_{i+2}D_{i+3}\ldots D_{i+C(\ln n)+1})=
\left(\frac{1}{\kappa}\right)^{C\ln n}$$
To see why this holds, simply note that $D_{i+2}$ is independent of $D_{i+1}$, so  $w$ and $v$
 agree on the first letter with probability  $1/\kappa$.
Then, $D_{i+3}$ is independent of $D_{i+2}$ and $D_{i+1}$, so
the second letter of $v$ has a probability of $1/\kappa$ to be identical
to the second letter of $w$, and so on.
Thus, to get that $whp$ no identical subsequence  appears in any two different positions in $D$, we need to get a suitable upper bound of the right side of (\ref{star}). If we take  $i$ and $j$ both in $n^\alpha$, we find that the probability to have at least
one subsequence appearing in two different positions in  $D$, can be bounded in the following way:
\begin{equation}
\label{starstar}
P(\cup_{i\neq j}E^{n+}_{i,j}) \leq  \sum_{i\neq j}P(E^{n+}_{i,j}) \leq n^{2\alpha}\cdot n^{-C\ln(\kappa)}
\end{equation}
The same type of bound also holds for $P(E^{n-}_{i,j})$ using (\ref{wordswords1}).
We take $C$ strictly larger than $2\alpha/\ln \kappa$
in order to get the right side of (\ref{starstar})  be negatively polynomially
small in $n$.

For our algorithm $\Lambda^n$ we will take subsequences (the short words) to have size $(\ln n)^2$ instead of being
linear in $\ln n$. Thus we do not  have to think about the constant in front of 
$\ln n$. Then, the bound we get for the probability becomes even better than
negative polynomial in $n$. It becomes of the type $n^{-\beta n}$ where $\beta>0$
is a constant independent of $n$.

\subsubsection{Applying the DNA-sequencing method to a multidimensional scenery}
In our algorithm $\Lambda^n$  we do not have the task to reconstruct a sequence,
instead we have to reconstruct a multidimensional scenery restricted to a box. 
We use the same idea as the DNA-sequencing method except that we will reconstruct several
long words instead of reconstructing just one long word  (one sequence).
This corresponds to the second phase of $\Lambda^n$ where
with a bag of short words, it constructs a collection
of long words.  These words will be the different
restrictions of the scenery $\xi$ to straight lines-segments parallel to some  direction of coordinates. This long words of course do not yet tell us exactly
how the scenery looks,  we will  still need to position these long words correctly
with respect to each other. (This is done in the fourth phase of $\Lambda^n$). Let us explain that with another example.

\begin{example}
Let us assume $d=2$ and  the 
scenery $\xi$ restricted to the box $[0,4]\times[0,4]$ be given by:
\begin{equation}
\label{thescenery1}
\xi_{[0,4]\times[0,4]}=
\begin{array}{c|c|c|c|c}
1&9&4&3&7\\\hline
5&0&7&6&1\\\hline
4&3&9&1&2\\\hline
6&1&4&0&4\\\hline
2&7&8&0&3
\end{array}
 \end{equation}
Here we assume that the alphabet has size $\kappa=10$ and that we would be given a bag of all short words
of size $4$ appearing in (\ref{thescenery1}).
That is words which can be read horizontally or vertically
in either direction: from up to down, down to up,  left to right, right to left and of size 4.
This bag of short-words is:
\begin{eqnarray}
\label{SHORT}
&\{&1943,5076,4391,6140, 27803,
9437,0761,3912,1404,7803,\\\nonumber
& & 1546,9031,4794,3610,7124,
5462,0317,7948,6100,1243
\},
\end{eqnarray}
and their reverses.
As assemble rule we use that words must coincide on at least $3$
consecutive letters. We can for example assemble $1943$ with $9437$
to get $19437$. Similarly we  assemble $1546$ with $5462$ and obtain
$15462$. So, we apply the DNA-puzzling trick, but instead of 
reconstruct only one long word, we will reconstruct several long words.
In this example we reconstruct the set of $10$ long words and their reverses:
\begin{eqnarray}
\label{LONG}
&\{&19437,50761,43912,61404,27803,\\\nonumber
& & 15462,90317,47948,36100,71243\},
\end{eqnarray}
where again each of the above long words could be its own reverse.
\end{example}

Thus, the previous example shows  how the second phase of the algorithm works:
From a set of shorter words $SHORTWORDS^n$,  we obtain a set of longer words 
$LONGWORDS^n$ in the same manner  how we obtained the set (\ref{LONG}) from the bag of words (\ref{SHORT}).  Note that the long words are successfully reconstructed in this example, because in the restriction of the scenery in (\ref{thescenery1}), any word of size $3$ appears at most in one place.

The differences in the numeric example presented above and how the second phase of  $\Lambda^n$ works are the following: the short words in $\Lambda^n$ have length $(\ln n)^2$ and the long words have length $4n$, instead of 4 and 5. Moreover, $\Lambda^n$ will need to  assemble in the second phase  many short words to get one long word, despite in this example where we just used two short words to get  each long word.
On the other side, in the previous example  is given to us a bag of  all words 
of size $4$ of the restriction of $\xi$ to the box $[0,4]\times[0,4]$.  However,  the second phase of $\Lambda^n$ has the  bag of short words $SHORTWORDS^n$,  which is not exactly equal to all words  of size $(\ln n)^2$ of  $\xi$ restricted to a box.  Instead it is the bag of words that contains
all words  of size $(\ln n)^2$ of $\xi$, restricted to the box $\mathcal{K}(4n)$, but  augmented
by some other words of the bigger box $\mathcal{K}(n^2)$.

The bag of  short words $SHORTWORDS^n$ is obtained in the first phase of $\Lambda^n$. The reason why the first phase is not able to identify which words are in the box $\mathcal{K}(4n)$ and which are in  $\mathcal{K}(n^2)$ is as follows: the observations in the first phase of  $\Lambda^n$ are taken up to time $n^2$. Since we assume that the first particle 
starts at time $0$ at the origin, by time $n^2$ all particles must be contained
in the box $\mathcal{K}(n^2)$,  and the probability for one given particle
at time $n^2$ to be close to the border of $\mathcal{K}(n^2)$
is exponentially small. Since we have many particles, a few will
be close to the border of $\mathcal{K}(n^2)$ by time
$n^2$, and these will be enough particles to provide some few words
close to the border of $\mathcal{K}(n^2)$ by time $n^2$, and  selected in the first phase of the algorithm.

There is an important consequence to this in the second phase of the algorithm,  since
some reconstructed long words in $LONGWORDS^n$ might also be ``far out in the box 
$\mathcal{K}(n^2)$'' and not in $\mathcal{K}(4n)$.

\subsubsection{The diamond trick to reconstruct all the words
of $\xi_{\mathcal{K}(4n)}$}\label{diamond}
In the previous subsection, we showed how to assemble shorter words
to get longer ones, but we have not yet explained the main
idea of how we manage to obtain  short words in the first phase of $\Lambda^n$, being given only the observations.  
The basic idea is to use the {\it diamonds associated with a word
appearing in the scenery}. Let us look at an example.

\begin{example}
Take  the following piece of a two dimensional scenery
which would be the restriction of $\xi$ to the $[0,6]\times[0,4]$:
\begin{equation}
\label{thescenery}
\xi_{[0,6]\times[0,4]}=
\begin{array}{c|c|c|c|c|c|c}
2&1&9&{\color{blue!80!black}4}&3&7&4\\\hline
7&5&{\color{blue!80!black}0}&{\color{blue!80!black}7}&{\color{blue!80!black}6}&1&1\\\hline
7&{\color{green!50!black}4}&{\color{green!50!black}3}&{\color{green!50!black}9}&{\color{green!50!black}1}&
{\color{green!50!black}2}&1\\\hline
8&6&{\color{blue!80!black}4}&{\color{blue!80!black}4}&{\color{blue!80!black}0}&4&3\\\hline
2&2&7&{\color{blue!80!black}8}&0&3&9
\end{array}
 \end{equation}

Consider the word  $${\color{green!50!black}w=43912}
=\xi_{(1,2)}\xi_{(2,2)}\xi_{(3,2)}\xi_{(4,2)}\xi_{5,2}$$ which appears in the above piece
of scenery  in green. That word ``appears between the points
$x=(1,2)$ and $y=(5,2)$''. We only consider here words which are written 
in the scenery ``along the direction of a coordinate''.
In blue we highlighted the {\it diamond associated with the word}
$43912$. More precisely the diamond consists of all positions
which in the above figure are green or blue.  
\end{example}

The formal definition is that if $x$ and $y$ are two points in $\mathbb{Z}^d$ so
that $\bar{xy}$ is parallel to a canonical vector $\vec{e}$,
then the diamond associated with the word 
$w=\xi_x\xi_{x+\vec{e}}\xi_{x+2\vec{e}}\ldots \xi_{y-\vec{e}}\xi_{y}$
consists of all points in $\mathbb{Z}^d$ which can be reached
with at most $(|x-y|/2)-1$ steps from the point $(y-x)/2$.
We assume here that the Euclidean distance $|x-y|$ is
an odd number.

The useful thing will be that $whp$  for a given
non-random point $z$ outside the diamond associated with $w$, there is no nearest neighbor walk path  starting at $z$ and generating as observations  $w$. So, $whp$
a word $w$ in the scenery can only be generated as observations  by a 
nearest neighbor-walk path starting (and also ending) in the diamond associated with $w$. (At least if we restrict the scenery to a box of polynomial size in the length of $w$).
To see this take for instance the following path $R$:
$$((0,0),(1,0),(1,1),(1,2),(2,2))$$
In the previous  example  a random walker which would  follow
this little path would observe the sequence of colors given by
\begin{equation}
\label{xicircR}
\xi\circ R=(2,2,6,4,3)
\end{equation}

Note that  the path $R$ and the straight path from $x$ to $y$
intersect.  So,  (\ref{xicircR}) and the green word ${\color{green!50!black}w=43912}$ are not independent of each other. However, because the path $R$ starts outside the diamond, we get  the probability of the event  that (\ref{xicircR}) and the word  $\color{green!50!black}w$ are identical, has the same probability as  if they would be independent. That is assuming that $R$ is a  (non-random) nearest neighbor path starting (or ending) at a given point $z$ outside the diamond  associated with  a word $w$,
$$P(\xi\circ R=\xi_x\xi_{x+\vec{e}}\xi_{x+2\vec{e}}\ldots \xi_{y-\vec{e}}\xi_{y})=\left(\frac{1}{\kappa}\right)^m,$$ where $m$ designates the size of the word.
In fact looking at our example, $R$ starts at $(0,0)$ which is
outside the diamond. So, the starting point of $R$ 
is different from  $x=(1,2)$ and hence by independence
$$P(\xi(R(0))=\xi_x)=\frac{1}{\kappa}.$$
Then the second letter in $w$, that is $\xi_{x+\vec{e}}$
 independent of the first two letters of the observations along the path,
$\xi(R_0)\xi(R_1),$
since both points $R(0)=(0,0)$ and $R(1)=(1,0)$
are different from $x+\vec{e}=(2,2)$. 
Thus, we get the probability that
the first two letters of  $w$ coincide with the
first two observations made by $R$ is equal to:
\begin{equation}\label{Prob2}
P(\xi(R(0))\xi(R(1))=\xi_x\xi_{x+\vec{e}})=\left(\frac{1}{\kappa}\right)^2.
\end{equation}
The proof goes on by induction:  the $k$th letter in the word
 $w_{x+k\vec{e}}$  is independent of 
the first $k$ observations made by $R$. The reason is that
the first $k$ positions $R(0)R(1)\ldots R(k-1)$
visited by $R$ do never contain the point $x+k\vec{e}$, 
since ``the walker following the path $R$ never catches up
with the walker going straight from $x$ to $y$.''

On the other hand observe that in our example,
we see a path starting inside the diamond and producing
as observation the same  green word  $\color{green!50!black}w$. Take the path
$$(2,1),(2,2),(3,2),(4,2),(5,2).$$
Thus, this ``counterexample''  illustates how ``easy'' it is 
for a path starting ``inside'' the diamond associated with a word $w$, 
to generate the same word.  This a second  path different to  the straight path ``going from $x$ to $y$''.

Now, using (\ref{Prob2}) we can calculate an upper bound  for the probability that for a given non-random point $z$ outside the diamond, there exists at least one non-random
path $R$ starting at $z$ and producing as observation $w$.
Observe that in  $d$-dimensional scenery, for a given starting
point $z\in \mathbb{Z}^d$, there are  $(2d)^{k-1}$ nearest neighbor walk paths
of length $k$.  So, we find that the probability that there exists a nearest neighbor path $R$
starting at $z$, with $z$ outside the diamond associated with a word 
 $w$, and  $R$ generating $w$,  has  a probability 
bounded from above as follows:
\begin{equation}
\label{frac}
P\left(\exists \;\text{a nearest neighbor path $R$ starting at $z$ with 
}\xi\circ R=w \right)\leq \left(\frac{2d}{\kappa}\right)^{k-1}
\end{equation}

Note that the bound above is negatively exponentially small
in $k$ as soon as $2d<\kappa$. The inequality $2d<\kappa$
is precisely the inequality given in Theorem \ref{maintheorem} which makes our
reconstruction algorithm work).

Thus, in the next section we will define the event $B^n_3$ as follows:
Let $w$ be a word of size $(\ln n)^2$ in $\xi_{\mathcal{K}(n^2)}$, and 
$R$ a nearest neighbor  path, $R:[0,k-1]\rightarrow \mathcal{K}(n^4),$ so that
$\xi\circ R= w$, and $R$ begins and ends in
the diamond associated with $w$. 

Note then that $P(B^{nc}_3)$   is bounded from above by 
(\ref{frac}) times the number of points in the box $\xi_{\mathcal{K}(n^2)}$. But expression (\ref{frac}) is negatively exponentially small in the size of the word $(\ln n)^2$
and hence dominates the polynomial number of points in the box, i.e., it goes to zero as $n$ goes to infinity.  Thus,  $B^n_3$ holds $whp$.  (To see the exact proof, go to lemma 3.3.)

Now, we know that with high probability the words can only
be generated by a nearest neighbor walk path starting (and ending) in
the diamond associated with $w$. (At least when we restrict ourselves to 
a box of polynomial size in the length of $w$). But,  how can we use this 
to reconstruct words?  The best is to consider an example.

\begin{example}
For this let the restriction of the scenery $\xi$ to $[0,16]\times[0,4]$ be equal to
\begin{equation}
\label{thescenery2}
\xi_{[0,16]\times[0,4]}=
\begin{array}{c|c|c|c|c|c|c|c|c|c|c|c|c|c|c|c|c}
2&1&9&{\color{blue!80!black}4}&3&7&4&1&2&5&2&2&7&{\color{red}8}&0&6&9\\\hline
7&5&{\color{blue!80!black}0}&{\color{blue!80!black}7}&{\color{blue!80!black}6}&1&1&8&2&5&8&6&{\color{red}7}&{\color{red}4}&{\color{red}0}&4&2\\\hline
7&{\color{green!50!black}4}&{\color{green!50!black}3}&{\color{green!50!black}9}&{\color{green!50!black}1}&
{\color{green!50!black}2}&1&7&8&4&7&{\color{brown}6}&{\color{brown}1}&{\color{brown}7}&{\color{brown}7}&{\color{brown}7}&4\\\hline
8&6&{\color{blue!80!black}4}&{\color{blue!80!black}4}&{\color{blue!80!black}0}&4&3&5&3&6&7&5&{\color{red}1}&{\color{red}9}&{\color{red}9}&9&1\\\hline
2&2&7&{\color{blue!80!black}8}&0&3&9&4&3&7&2&1&9&{\color{red}4}&5&7&0
\end{array}
 \end{equation}
Let the word which appears in green
be denoted by $w_1$ so that 
$$w_1={\color{green!50!black}43912}$$
Let the word written in brown be denoted by $w_3$ so that
$$w_3={\color{brown} 61777}$$
Finally let the word which is written when we go straight from
the green word to the brown word be denoted by
$w_2$ so that
$$w_2=17847$$
In the current example the diamond $D_1$ associated with the green word
$w_1$ is given in blue and the diamond $D_3$ associated with the 
brown word $w_3$ is highlighted in red.
Note that there is only one shortest nearest neighbor path to go from $D_1$ to $D_3$,
walking straight from the point $(5,2)$ to the point $(11,2)$ in 
exactly six steps. There is not other way to go in six  steps
from $D_1$ to $D_3$. When doing so a walker will see the word $w_2$.
Now assume that the size of our short words is $5$, (that is the
 size which in the algorithm is given by the formula $(\ln n)^2$). Assume also that the rectangle
$[0,16]\times[0,4]$ is contained in $\mathcal{K}(n^2)$.

Remember that  if $B^n_3$ holds, we have that within the box $\mathcal{K}(n^2)$
a nearest neighbor walk can only generate a short word of 
$\xi_{\mathcal{K}(n^2)}$ if it starts and ends in the diamond associated
with that word. Using this to the words $w_1$ and $w_3$,  we see in the observations
the following pattern
$$w_1*****w_3$$
where $*$ is a wild card which stands for exactly one letter,
then $whp$ the walker between $w_1$ and $w_3$ was walking in a straight manner
from $D_1$ to $D_3$. Hence, the wild card sequence $*****$ must then be 
the word $w_2$. Of course, we need at least one particle to follow
that path in order to observe $w_1w_2w_3$ up to time
$n^2$. This will be taken care in $\Lambda^n$  by the event $B^n_2$ which stipulates
that any nearest neighbor path of length $3(\ln n)$ contained
in $\mathcal{K}(4n)$, will be followed by at least one particle up
before time $n^2$.  In other words we have proven that if
$B^n_2$ and $B^n_3$ both hold, then $w_2$ gets selected by the first
phase of the algorithm as a short word of $SHORTWORDS^n$.
The argument of course works for any short word of 
$\xi_{\mathcal{K}(4n)}$ and hence we have that
$$B^n_2\cap B^n_3 \implies W(\xi_{\mathcal{K}(4n)})\subset SHORTWORDS^n
$$
\end{example}

Thus the previous example  shows how the first phase of the algorithm manages to reconstruct 
all short words in $\xi_{\mathcal{K}(4n)}$. 

\subsubsection{How to eliminate junk observation-strings which are not words of 
$\xi_{\mathcal{K}(n^2)}$}
In the previous subsection we have shown how to reconstruct enough words.
But now the next question is ``how do we manage to not reconstruct too many
words''? By this we mean, how do we make sure that observations
which do not correspond to words of $\xi_{\mathcal{K}(n^2)}$
do not get selected by the first phase of our algorithm?
This means that we have to be able to eliminate observations
which do not correspond to  a word of $\xi_{\mathcal{K}(n^2)}$!
The best is again to see an example. 
\begin{example}
Take for this
\begin{equation}
\label{thescenery3}
\xi_{[0,16]\times[0,4]}=
\begin{array}{c|c|c|c|c|c|c|c|c|c|c|c|c|c|c|c|c}
2&1&9&{\color{blue}4}&3&7&4&1&2&5&2&{\color{red}2}&7&8&0&6&9\\\hline
7&5&{\color{blue}0}&{\color{blue}7}&{\color{blue}6}&1&1&8&2&5&{\color{red}8}&{\color{red}6}&{\color{red}7}&4&0&4&2\\\hline
7&{\color{green!50!black}4}&{\color{green!50!black}3}&{\color{green!50!black}9}&{\color{green!50!black}1}&
{\color{green!50!black}2}&1&7&8&{\color{brown}4}&{\color{brown}7}&{\color{brown}6}&{\color{brown}1}&{\color{brown}7}&7&
7&4\\\hline
8&6&{\color{blue}4}&{\color{blue}4}&{\color{blue}0}&4&3&5&3&6&{\color{red}7}&{\color{red}5}&{\color{red}1}&9&9&9&1\\\hline
2&2&7&{\color{blue}8}&0&3&9&4&3&7&2&{\color{red}1}&9&4&5&7&0
\end{array}
 \end{equation}
Observe this is the same piece of scenery as was shown in (\ref{thescenery2}), but  the brown word was moved two units to the left.
So, let again $w_1$ denote the green word
$${\color{green!50!black}w_1=43912}$$
and let this time $w_3$ be the ``new'' brown word:
$${\color{brown} w_3=47617}$$
Now a particle in between following the green word and then the brown word
could for example do the following little dance step:
$$right,up,right,down, right,right,$$
and then produce
the observation string $w_2$ given
by
$$w_2=11878.$$
How can we make sure the observation string $w_2$ which is not a word (it does not follow a straight path) of our piece of scenery, does not get selected by the first phase of our algorithm as a short word?
To see how $w_2$ gets eliminated in the first phase of our algorithm consider the following
dancing step:
$$right,down,right,up,right,right.$$
When doing this nearest neighbor path, a particle would 
produce the observations $\bar{w}_2$ where
$$\bar{w}_2=13578$$ 

Let $B^n_5$ be the event that up to time $n^4$ every nearest neighbor path
of length $3(\ln n)^2$ in $\mathcal{K}(n^2)$ gets followed at least once.
Then, assuming our piece of scenery
 (\ref{thescenery3}) is contained in $\mathcal{K}(n^2)$, we would have that:
Up to time $n^4$ we will observe
both strings
$${\color{green!50!black} w_1}w_2{\color{brown}w_3}  \quad\quad\text{ and }\quad\quad {\color{green!50!black} w_1}\bar{w}_2{\color{brown}w_3},$$
at least once. 
Since $w_2\neq\bar{w}_2$ the second short-word-selection
criteria of the first phase assures that the words $w_2$ and $\bar{w}_2$ do not get selected
in the first phase of $Lambda_n$. The crucial thing
here was that from the point $(5,2)$ to $(8,2)$
there were two different $6$ step nearest neighbor paths
generating different observations, $w_2$ and $\bar{w}_2$.
In  subsection \ref{s_1phase}, we will show that $whp$
in  the box $\mathcal{K}(n^2)$ for any pair of points $x$ and $y$
so that a nearest neighbor walk goes in $(\ln n)^2-1$ steps from 
$x$ to $y$, either one of the two following things hold:
\begin{enumerate}
\item The segment $\bar{xy}$ is parallel to a direction of a coordinate
and the distance $|x-y|=(\ln n)^2-1$, or 
\item there exist two different nearest neighbor walk paths
going from $x$ to $y$ in $(\ln n)^2-1$ steps and generating to
different observation-strings.
\end{enumerate}
\end{example}
So, this  implies that $whp$, the first phase of our algorithm
can eliminate all the strings as it does in the example with $w_2$ and $\bar{w}_2$  which are not words
of $\xi_{\mathcal{K}(n^2)}$.

\subsubsection{Why we need a seed}
\label{seed}
The second phase produces a bag of long words which $whp$
contains all long words of $\xi_{\mathcal{K}(4n)}$. Unfortunately, this bag
is likely to contain also some long words of $\xi_{\mathcal{K}(n^2)}$
which are ``not close to the origin''. 

Note that the size of the long words is $4n$, so if such a long word appears close to the border
of $\mathcal{K}(n^2)$, then it wouldn't serve to our reconstruction
purpose.
The reason is that the algorithm $\varLambda^n$
aims to reconstruct the scenery in a box close to the origin.
That is why in the third phase of $\varLambda^n$ we apply the first two 
phases  but with the parameter $n$ being replaced
by $n^{0.25}$. We then take any long word $w_0$ from the bag
of words produced by the second phase of the algorithm
$\varLambda^{n^{0.25}}$. That long word $w_0$ of size $4n^{0.25}$
is then $whp$ contained in  $\xi_{\mathcal{K}((n^{0.25})^2)}=\xi_{\mathcal{K}(n^{0.5})}.$

In other words, the long word $w_0$ chosen as seed in the third phase of the 
algorithm is $whp$ not further away than $\sqrt{n}$
from the origin. Since it is likely that any word of that size
appears only once in $\mathcal{K}(n^2)$, then  we can use this to  determine
one long word $w_L$ from the bag created by $\varLambda^n$ 
which is not further from the origin than $ \sqrt{n}$. We simply chose $w_L$ to be any word from the bag of long
words created by $\varLambda^n$ which contains $w_0$.
Finally in the fourth phase of the algorithms we will then add neighboring 
long words to that first chosen long word. If the one long word
which gets chosen in the third phase is close to the origin,
we can determine which long words are neighbor on each other,
then this will ensure that the other long words used for the reconstruction
in the fourth phase are also close to the origin.

\subsubsection{Finding which long words are neighbors in the 4th phase of 
the algorithm $\Lambda^n$.}\label{Lamblamb}

The fourth phase of the algorithm is then concerned with finding
the relative position of the longer reconstructed words to each other.
More precisely it tries to determine which long words are neighbors
of each other. For the exact definition of neighboring long words
check out subsection \ref{s_4phase}. Let us explain it through another example.

\begin{example}
Consider the piece of scenery $\xi_{[0,16]\times[0,4]}$ given
in (\ref{thescenery3}) and  let us designate by $v_a$ the green word, by $v_c$ the brown word and by $v_b$ the   word between $v_a$ and $v_c$, i.e., 
$$v_a=\xi_{(1,2)}\xi_{(2,2)}\xi_{(3,2)}\xi_{(4,2)}\xi_{(5,2)}=43912,$$
$$v_c=\xi_{(9,2)}\xi_{(10,2)}\xi_{(11,2)}\xi_{(12,2)}\xi_{(13,2)}=47617,$$
and 
$$v_b=\xi_{(6,2)}\xi_{(7,2)}\xi_{(8,2)}=178.$$
Finally let $w_b$ designate the word ``one line higher'' than $v_b$, so 
$$w_b:=\xi_{(5,3)}\xi_{(6,3)}\xi_{(7,3)}\xi_{(8,3)}\xi_{(9,3)}=11825.$$

Note  that $w_b$ has two digits more than $v_b$, and the middle letter of $w_b$ has the same $x$-coordinate than the last  letter of  $v_b$ in (\ref{thescenery3}). Furthermore,  in the piece of scenery (\ref{thescenery3}) we designate the third line
by $v$ so that 
$$v:=\xi_{(0,2)}\xi_{(1,2)}\xi_{(2,2)}\ldots\xi_{(16,2)}=74391217847617774,$$
and by $w$ the long word written one line above, i.e.,
$$w:=\xi_{(0,3)}\xi_{(1,3)}\xi_{(2,3)}\ldots\xi_{(16,3)}=75076118258674042.$$
Assume that the two words 
$v$ and $w$ have already been reconstructed by the two first phases
of our algorithm $\Lambda^n$.  Consider next the straight path $R_1$:
\begin{eqnarray*}
& & (1,2),(2,2),(3,2),(4,2),(5,2),(6,2),(7,2),\\
& & (8,2),(9,2),(10,2),(11,2),(12,2),
(13,2),
\end{eqnarray*}
so a particle following this straight path generates the observations $\xi\circ R_1=v_av_bv_c.$

Consider now a second path $R_2$ that is similar to $R_1$ but at the end of the word $v_a$, it goes one step up to read $w_b$ and then, at the end of $w_b$ it goes one step down to
read $v_c$. So, the path
$R_2$ is defined as follows:
\begin{eqnarray*}
& &  (1,2),(2,2),(3,2),(4,2),(5,2),
(5,3),(6,3),(7,3),\\
& & (8,3),(9,3),(9,2),(10,2),(11,2),(12,2),(13,2),
\end{eqnarray*}
and generates the observations $\xi\circ R_2=v_aw_b v_c.$

If now up to time $n^4$ there is at least one particle following
the path $R_1$ and another particle following $R_2$,
then in the observations before time $n^4$, we will
observe once $v_av_bv_c$ and also $v_aw_bv_c$.
So, $v_a$, $v_b$, $v_c$ and $w_b$ pass all the criteria in 
the fourth phase of our algorithm together with the long words
$v$ and $w$. So, $v$ and $w$ are detected (correctly)
as being neighboring long words.  Again, for this to happen we only
need particles to go on the path $R_1$ and $R_2$.
\end{example}

So, what the previous example tell us is that to recognize correctly which words in 
$LONGWORDS^n$ are  neighbors in the fourth phase of our algorithm, we need first  to guarantee  that  all nearest neighbor paths of length 
$3(\ln n)^2$ within the box $\mathcal{K}(n^2)$, being followed by at least
one particle up to time $n^4$. The event $B^n_3$ guaranties this
and we prove it holds $whp$ in the first subsection of the next
section.

On the other side, we still need  the forth phase of $\Lambda^n$  not to classify pairs of long words as neighbors if they are not  in the scenery. This problem of  ``false positives'' is solved as soon as the previously defined diamond property holds, 
that is when the event $B^n_3$ holds and also short words
do not appear in different places in $\xi_{\mathcal{K}(n^2)}$.
The proof of this is a little intricate and is given
as proof of Lemma \ref{4phaseworks}
in the next section. 

\subsubsection{Placing neighboring long words next to each other in 
the $4$-th phase of our algorithm.}\label{library}
Here we show  how the long words are located next to each other in
the fourth phase, in order to reconstruct a piece of the scenery
$\xi$. 
\begin{example}
Let us assume that the scenery is $2$-dimensional and that the seed word is $w_0=012$. To start we place $w_0$ at the origin,
and it would be  the first part of the
 ``reconstruction done by the algorithm $\Lambda^n$'',  thus $w_0$ is  the restriction of $\Lambda^n(\chi_{n^4})$ to
the set $\{(-1,0),(0,0),(1,0)\}$.

Say then that we find a long word $w_L$ which is $60123$. This long word
contains the seed word $012$ centered in its middle, so, if we  superpose  $w_L$ over $w_0$ so that it matches we   get again $60123$,
that  corresponds to the restriction of 
$\Lambda^n(\chi_{n^4})$ to
the set $\{(-2,0),(-1,0),(0,0),(1,0),(2,0)\}$.

Next we find that the two long words $01111$ and $02222$
of $LONGWORDS^n$ are neighbors of $w_L$. In that case,
we place these two long words in a neighboring position to
get the following piece of scenery:
$$
\begin{array}{ccccc}
01111\\
60123\\
02222,\\
\end{array}
$$
and  assume that in our bag $LONGWORDS^n$ of long words
we find the word $03333$ to be a neighbor of $01111$ and $04444$ to be neighbor
of $02222$. Then our reconstruction yields the following piece
of scenery:
\begin{equation}
\label{finalfinal}
\begin{array}{ccccc}
03333\\
01111\\
60123\\
02222\\
04444
\end{array}
\end{equation}
This is then the ``end product'' produced by the algorithm $\Lambda^n$.
So, this final output of the algorithm  
is  a piece of scenery on a box of $5\times 5$ given in (\ref{finalfinal}).
We assumed the size of the long words to be $5$.
\end{example}

\subsubsection{Overview over the events which make the algorithm 
$\varLambda^n$
work.}
A handful of events ensure that the algorithm $\varLambda^n$
works as already mentioned. These events will all be defined
again in a more exhaustive way in the next section. But, here let us list them 
in a somewhat informal way. 
The {\bf first phase} works through the following events: 
\begin{itemize}
\item{}The event
$$B^n_2$$
guaranties that up to time $n^2$ every nearest neighbor path of length
$3(\ln n)^2$ contained in $\mathcal{K}(4n)$ gets followed at least once
by at least one particle.
\item{}The event
$$B^n_3$$ asserts that up to time $n^4$ for every word $w$ of $\xi_{\mathcal{K}(n^2)}$,
any nearest neighbor walk path $R$ in $\mathcal{K}(n^4)$ generating
$w$ must start and end in the diamond associated with $w$. This event
is needed to guaranty that all the  words of $\xi_{\mathcal{K}(4n)}$
of length $(\ln n)^2$
get selected in the first phase of the algorithm, and put into the 
bag of words $SHORTWORDS^n$.
\item{}The event
$$B^n_4$$ states that for any two points $x$ and $y$ in $\mathcal{K}(n^2)$
so that there is a nearest neighbor path going in $(\ln n)^2-1$ steps
from $x$ to $y$, either one of the following  holds:
\begin{itemize}
\item{}The points $x$ and $y$ are at distance $(\ln n)^2-1$ from
each other and the segment $\bar{xy}$ is parallel to a direction
of coordinate. In that case, there can  only be one nearest neighbor walk
in $(\ln n)^2-1$ steps from $x$ to $y$, and that nearest neighbor walk
goes ``straight''. 
\item{}There exist two different nearest neighbor walk paths going
from $x$ to $y$ in exactly $(\ln n)^2-1$ steps and generating
different observations.
\end{itemize}
\item{}The event
$$B^n_5$$
guaranties that up to time $n^4$ every nearest neighbor path of length
$3(\ln n)^2-1$ contained in $\mathcal{K}(n^2)$ gets followed at least once. Furthermore, 
together with the event $B^n_4$,  these  make sure that observations
which are not words of $\xi_{\mathcal{K}(n^2)}$ get eliminated in the first
phase of $\varLambda^n$.
\end{itemize}
In Lemma \ref{B}, the combinatorics of the {\bf first phase} is proven.
It is shown that when  all the events $B^n_2,B^n_3,B^n_4, B^n_5$
hold, then the first phase works. In the first subsection
of the next section it is also proven that all these events
hold $whp$, which implies that the first phase works $whp$.

For the {\bf second phase} of the algorithm, we only need that 
 any  word of $\xi_{\mathcal{K}(n^2)}$ of size $(\ln n)^2-1$ 
appears in only one place in $\mathcal{K}(n^2)$.  That is given by the  event $C^n_1$ and needed to assemble short words into longer words. 

The {\bf third phase} of the algorithm is just using the first two phases of the algorithm
but with a parameter different from $n$.  Instead the parameter is $n^{0.25}$.
So, we don't need any other special event to make this phase work.
We only need what we have proven for the first two phases of the algorithm.

Finally, the {\bf forth phase} of the algorithm  needs the diamond property to hold for the short words in $\xi_{\mathcal{K}(n^2)}$,
that is the event $B^n_3$ to hold.  Furthermore, the event $C^n_1$ which guaranties that short words can not appear in two different places of $\xi_{\mathcal{K}(n^2)}$ is also needed. These events are defined already for the first two phases of the algorithm.

The next section gives all the detailed definitions of these events, the proofs for their high probability, and also  the rigorous proofs that these events make the different phases
of the algorithm work. Although  most of the  main ideas are already given 
in the present section, and it might seem a little redundant, we feel
that presenting them once informally but with all the details 
in a rigorous manner,  will be very useful to understand better the algorithm. 

\section{Proof of  Theorem \ref{n-theorem}}
In what follows we will 
we say that the Branching random walk BRW
visits $z\in\Z^d$ at time $t$ if $\eta_t(z)\geq 1$.


\subsection{First phase}
\label{s_1phase}
In this phase we will construct  the set of $SHORTWORDS^n$.
Recall that  a string $w_2$ of size $(\ln n)^2$ is in
$SHORTWORDS^n$ if there exist two sequences
$w_1$ and $w_3$ both of size $(\ln n)^2$ and such that:
\begin{enumerate}
\item{} $w_1w_2w_3$ appears in the observations before time
$n^2$. 
\item{} The only string $w$ of size $(\ln n)^2$ such that
$w_1ww_3$ appears in the observations up to time $n^4$ is $w_2$.
\end{enumerate}
Let $W(\xi_{\mathcal{K}(4n)})$ and $W(\xi_{\mathcal{K}(n^2)})$ be
the sets of all  words of size $(\ln n)^2$ in
$\xi_{\mathcal{K}(4n)}$ and $\xi_{\mathcal{K}(n^2)}$ respectively,
then we are going to show that with high probability
 the set   $SHORTWORDS^n$ satisfies that 
$$W(\xi_{\mathcal{K}(4n)})\subseteq SHORTWORDS^n\subseteq 
W(\xi_{\mathcal{K}(n^2)}).$$
We need the following results:
\begin{lemma}\label{B1}
Let $B^n_1$ be the event that up to time $n^2$ all the sites 
in ${\mathcal{K}(4n)}$ are visited by the BRW more than
$\exp(cn)$ times, where $c$ is a constant independent of $n$, then
there exists $C>0$ such that
$$P(B^n_1)\geq 1- e^{-Cn^2}.$$
\end{lemma}
\begin{proof}
We only sketch the proof:
 it is elementary to obtain that by time $n^2/2$ the
process will contain at least $e^{\delta n^2}$ particles (where
$\delta$ is small enough), with probability at least $1-
e^{-Cn^2}$. Then, consider any fixed $x\in{\mathcal{K}}$; for
each of those particles, at least one offspring will visit~$x$
with probability at least $cn^{-(d-2)}$, and this implies the
claim of Lemma~\ref{B1}.
\end{proof}

\begin{lemma}\label{B2}
Let $B^n_2$ be the event that up to time $n^2$ for every nearest
neighbor
path of length $3(\ln n)^2-1$ contained in 
${\mathcal{K}(4n)}$, there is at least one particle which follows
that path.  Then, for all $n$ large enough we have:
$$P(B^n_2)\geq 1- \exp[-c_1n],$$
where $c_1>0$ is constant independent of $n$.
\end{lemma}
\begin{proof}
Let $R^z$ be a nearest neighbor path of length $3(\ln n)^2-1$ in 
$\mathcal{K}(4n)$ starting at  $z\in \mathcal{K}(4n)$. 
By Lemma~\ref{B1} we know that with high probability up to time $n^2$ all the
sites in ${\mathcal{K}(4n)}$ are visited by the BRW more than
$\exp(cn)$ times, where $c$ is a constant independent of $n$. 
Suppose  we have been observing the BRW up to time $n^2$, then
define for any $z\in \mathcal{K}(4n)$ the following variables;
$$
Y_{ij} = \left\{\begin{array}{cl} 1&\mbox{after the $i$-th
visit to $z$, 
the corresponding particle follows the path $R_j^z$ } \\
0&\mbox{otherwise,}\end{array}\right. 
$$
where $i=1,\dots,\exp(cn)$ and $j=1,\dots,(2d)^{3(\ln n)^2-1}$. 
 Note that the variables  $Y_{ij}$'s are independent because all
the particles are moving independently between themselves.
We are interested on the event 
\begin{eqnarray}
\label{event}
\Big\{\bigcap_{j=1}^{(2d)^{3(\ln n)^2-1}}
 \bigcup_{i=1}^{e^{cn}}\{Y_{ij}=1\}\Big\},
\end{eqnarray}
i.e, for every path of length $3(\ln n)^2-1$ starting at
 $z\in \mathcal{K}(4n)$, up to time $n^2$, there is at least one
visit to $z$, such that  the corresponding particle on $z$ follows it.

Let  $Z_j=\sum_i Y_{ij}$, thus $Z_j$ counts the number of times 
 $R_j$ is followed, and it is binomially distributed  with
expectation and variance given by
\begin{eqnarray*}
E[Z_j]&=&\exp(cn)\Big(\frac{1}{2d}\Big)^{3(\ln n)^2-1} \text{ and}
\\
V[Z_j]&=&\exp(cn)\Big(\frac{1}{2d}\Big)^{3(\ln n)^2-1} 
\Big(1-\Big(\frac{1}{2d}\Big)^{3(\ln n)^2-1}\Big).
\end{eqnarray*}
Observe that  (\ref{event}) is equivalent 
to $\Big\{\bigcap_{j=1}^{(2d)^{3(\ln n)^2-1}}\{Z_j\geq1\}\Big\}$,
then by Chebyshev's inequality we have
$$
P(Z_j\leq 0)\leq \frac{V[Z_j]}{E^2[Z_j]}<\exp[(3(\ln
n)^2-1)\ln2d-cn],
$$
and
\begin{eqnarray}
\label{B2c}
P\Big(\bigcup_{j=1}^{(2d)^{3(\ln n)^2-1}}\{Z_j\leq0\}\Big)
<\exp[(6(\ln n)^2-2)\ln2d-cn].
\end{eqnarray}
Since the number of sites in $\mathcal{K}(4n)$ is $(8n+1)^{d}$, 
by (\ref{B2c})  it follows that
\begin{equation}
\label{yesyes}
P(B^{nc}_2)<(8n+1)^{d}\exp[(6(\ln n)^2-2)\ln2d-cn].
\end{equation}
Now note that for any constant $0<c_1<c$,
the right side of  (\ref{yesyes}) 
is less than $\exp(-c_1 n)$, for $n$ large enough. Thus $P(B^{nc}_2)\rightarrow 0$ as $n\rightarrow\infty.$
\end{proof}

In what follows we will denote by $T_w$ the diamond 
associated with a word $w$ appearing in a certain place in
the scenery. For the definition of diamond associated 
with a word see (\ref{diamond}).

\begin{lemma}
\label{B3}
Let $B^n_3$ be the event that  for any word
$w$ of size $(\ln n)^2$  contained in
$\xi_{\mathcal{K}(n^2)}$,  every nearest neighbor walk path $R$
generating $w$ on $\xi_{\mathcal{K}(n^4)}$ must start and end  in the diamond
associated to $w$, i.e.,  $R(0)\in T_w$ and
$R((\ln n)^2-1)\in T_w$.  Then, 
$$
P(B^{n}_3)\geq 1-2^{2d+1} \exp\left(8d\ln(n+1)+(\ln n)^2\ln(2d/\kappa)\right),
$$
which goes to  $1$ as $n\rightarrow\infty$
because we have assumed $\kappa>2d$.
\end{lemma}
\begin{proof}
Take without loss of generality a sequence read  in a straight way 
from left to right, i.e.,
$$
w=\xi(x)\xi(x+\vec{e_1})\xi(x+2\vec{e_1})\ldots 
\xi(x+((\ln n)^2-1)\vec{e_1}),
$$
and let $R^z$ be a nearest neighbor walk (non-random) 
of length $(\ln n)^2-1$ star\-ting at $z$ with $z\in
\mathcal{K}(n^4)\setminus T_w$. Since a nearest neighbor path at
each unit of time only make steps of length one, then  it follows
that
$w_i$ is independent of
$\xi(R(0)),\xi(R(1)),\ldots,\xi(R(i-1))$,
as well as of $w_1,\ldots,w_{i-1}$. Hence,
\begin{equation}
\label{sumumu}
P[w=\xi(R(0))\xi(R(1))\ldots\xi(R((\ln n)^2-1))]=(1/ \kappa
)^{(\ln n)^2}
\end{equation}
(recall that~$\kappa $ is the number of colors). 

Let $P_z^n$ be the set of all nearest neighbor paths of length $(\ln n)^2-1$ starting at
$z$, and note that  $P_z^n$ contains no more than $(2d)^{(\ln n)^2-1}$ elements. For a fix $z\in\mathcal{K}(n^4)\setminus T_w$, let $B^n_{3,z}$ be the event that there is no nearest neighbor path $R^z$ of length $(\ln n)^2-1$ generating $w$. By (\ref{sumumu}),
 it follows that
\begin{eqnarray*}\label{P1}
P(B^{nc}_{3,z})&=&P[\exists R^z; \xi(R^z)=w]\\
&\leq&\sum_{R^z\in P_z^n} P(\xi(R^z)=w)\\
\label{P2}
&\leq&\frac{(2d)^{(\ln n)^2-1}}{\kappa
^{(\ln n)^2}}\\
&\leq&\left(\frac{2d}{\kappa}\right)^{(\ln n)^2}.
\end{eqnarray*} 
 
Now for any $z\in\mathcal{K}(n^4)\setminus T_w$, let $B^n_{3S}$ be the event that there is no nearest neighbor path $R^z$ of length $(\ln n)^2-1$  generating $w$. Hence,
\begin{eqnarray*}
\label{ecucu}P(B^{nc}_{3S})=P\Big(\bigcup_{z\in \mathcal{K}(n^4)
\setminus T_w}B^{nc}_{3,z}\Big)
&\leq& \sum_{z\in \mathcal{K}(4n)\setminus T_w}P(B^{nc}_{3,z})\\
&\leq& (2n^4+1)^{2d} \left(\frac{2d}{\kappa}\right)^{(\ln n)^2}\\
\end{eqnarray*}
Now by symmetry
$P(B^{nc}_3)\leq 2P(B^{nc}_{3S}),$
so that
\begin{eqnarray*}
P(B^{nc}_3)&\leq& 2(2n^4+1)^{2d} \left(\frac{2d}{\kappa}\right)^{(\ln n)^2}\\
&\leq& 2(2(n+1)^4)^{2d} \left(\frac{2d}{\kappa}\right)^{(\ln n)^2}\\
&=&\left(8d\ln(n+1)+(\ln n)^2\ln(2d/\kappa)\right).
\end{eqnarray*}
%
\end{proof}

\begin{lemma}\label{B4}
Let $B_4^n$ be the event that for every two points
$x,y\in\mathcal{K}(n^2)$ and such that there is a nearest neighbor
path $R$ of length $(\ln n)^2-1$ going from $x$ to $y$, 
one of the following alternatives 
 holds:
\begin{itemize}
 \item[(a)] $x$ and $y$ are at distance $(\ln n)^2-1$ and on the 
same line, that means, along the same direction of a coordinate, or
 \item[(b)] there exists two different nearest neighbor
 paths $R_1$ and $R_2$
of length $(\ln n)^2-1$ both going from $x$ to $y$ but 
 generating different observations,
i.e., $\xi(R_1)\neq \xi(R_2)$.
\end{itemize}
Then, we have 
$$
P(B_n^4)\geq 1-\exp[- 0.5(\ln n)^2\ln
\kappa ].
$$
\end{lemma}
\begin{proof} 
Let $x$ and $y$ be two points in  $\mathcal{K}(n^2)$ and such that
 there is a nearest neighbor path $R$ of length $(\ln n)^2-1$ going
from $x$ to $y$, so the distance between $x$ and $y$, $d(x,y)\leq
(\ln n)^2-1$. 

Suppose that  $d(x,y)= (\ln n)^2-1$, which is defined as 
 the length of the shortest path between $x$ and $y$.  
If $x$ and $y$ are not on the same  line,
then, 
there exists two paths $R_1$ and $R_2$ of length $(\ln n)^2-1$ both going from
$x$ to $y$ and not intersecting anywhere, except in $x$ and $y$. This means that $R_1(0)=R_2(0)=x$ and $R_1((\ln n)^2-1)=R_2((\ln n)^2-1)=y$,
but for all $j_1,j_2$ strictly between $0$ and $(\ln n)^2-1$,
we find $R_1(j_1)\neq R_2(j_2).$
Since the scenery $\xi$ is i.i.d, it thus follows that
\begin{eqnarray}\label{R1R2}
P[\xi(R_1(0))=\xi(R_2(0)),\dots,\xi(R_1((\ln n)^2-1))&=&\xi(R_2((\ln
n)^2-1))]\nonumber\\
&=&\Big(\frac{1}{\kappa }\Big)^{(\ln n)^2-2}.
\end{eqnarray}

Now suppose that  $d(x,y)<(\ln n)^2-1$.
Let $R_1$ be a path which makes a cycle from $x$ to $x$
 and then going in shortest time from $x$ to $y$, and $R_2$ be a
path which follow first a shortest path between $x$ and $y-1$, next
 makes a cycle from $y-1$ to $y-1$ and then go to $y$. If neither
the cycle from  $x$ to $x$  intersects the shortest path which makes
part of $R_2$ nor the cycle from $y-1$ to $y-1$ intersects the
shortest path which makes part of $R_1$, then we have that for
$i=0,\dots,(\ln n)^2-1$, the positions taken by $R_1$ and $R_2$
are different, i.e., $R_1(1)\neq R_2(1),\dots,R_1((\ln n)^2-2)\neq
R_2((\ln n)^2-2)$. Hence, $\xi(R_1(i))$ is independent of
$\xi(R_2(i))$ for $i=1,2,\ldots,(\ln n)^2-2$ and  (\ref{R1R2}) holds again.

Let $B_{4xy}^n$ be the event that there exist two nearest neighbor
paths~$S$ and~$T$ going from~$x$ to~$y$ with $d(x,y)\leq(\ln n)^2-1$, but generating
different observations, and let
$B^n_4=\bigcap_{x,y}B^n_{4xy},$
where the  intersection  is taken over
all $x,y\in \mathcal{K}(n^2)$ such that $d(x,y)\leq(\ln n)^2-1$.
By (\ref{R1R2}) it follows that 
\begin{equation}
\label{exp6}P(B_{4xy}^{nc})\leq \exp[- ((\ln n)^2-2)\ln \kappa ],
\end{equation}
then
\begin{eqnarray}
\label{susu}
P(B^{nc}_4)&\leq& \sum_{x,y}P(B^{nc}_{4xy})\nonumber\\
&<&[2(\ln n)^2-3]^d(2n^2+1)^d\exp[- ((\ln n)^2-2)\ln \kappa ]\nonumber\\
\label{Bn4}
&=&\exp[d\ln(2(\ln n)^2-3)+d\ln(2n^2+1)- ((\ln n)^2-2) \ln \kappa ],
\end{eqnarray}
and observe that  the right side of (\ref{Bn4}) is for all $n$ large enough
less than
$\exp(-0.5(\ln n)^2 \ln \kappa )$. This finishes this proof.  
\end{proof}

\begin{lemma}
\label{B5}
 Let $B^n_5$ be the event  that up to time $n^4$ every path of length
$3(\ln n)^2-1$ in $\mathcal{K}(n^2)$ gets followed at least once by a particle.
Then
$$P(B^n_5)\geq 1- \exp[-c_2n^2],$$
where $c_2>0$ is a constant not depending on $n$.
\end{lemma}
\begin{proof}
Note that the event $B^m_2$ from Lemma \ref{B2}, where we take $m=n^2$ 
gives that up to time $n^4$, every path in $\mathcal{K}(4n^2)$ of length
$12\ln n-1$ gets followed by at least one particle. So, this implies that $B^n_5$ holds
and hence
$$B^{n^2}_2\subset B^n_5.$$
The last inclusion above implies
\begin{equation}
\label{altanta}
P(B^n_5)\geq P(B^{n^2}_2).
\end{equation}
We can now use the bound from Lemma \ref{B2} for bounding
the probability of $P(B^{n^2}_2)$. Together with (\ref{altanta}),
this yields
\begin{equation}
\label{atlantic}
P(B^n_5)\geq 1- \exp[d\ln(8n^2+1)+(6(\ln n^2)^2-2)\ln 2d-cn^2].
\end{equation}
Now note that for any constant $c_2>0$ for which $c>c_2$,
we have that:\\
for all $n$ large enough we have
that the bound on the right side of inequality (\ref{atlantic})
is less than $\exp(-c_2 n^2)$ which finishes this proof.
\end{proof}\\[3mm]

\begin{lemma}\label{B}[The first phase works.]
Let $B^n$ designate the event that  every word of size
$(\ln n)^2$  in $\xi_{\mathcal{K}(4n)}$ is  contained in the set
$SHORTWORDS^n$, and also
all  the strings in  $SHORTWORDS^n$  belong to
$W(\xi_{\mathcal{K}(n^2)})$, then
$$ B_1^n\cap B_2^n\cap B_3^n\cap B_4^n\cap B^n_5\subset B^n.$$
\end{lemma}
\begin{proof} 
We start by proving that every word in $ W(\xi_{\mathcal{K}(4n)})$ of size $(\ln n)^2$, say $w_2$,
is contained in $SHORTWORDS^n$. Then there exist two integer points
$x,y\in\mathcal{K}(4n)$ 
on the same  line, i.e., along the same direction of a coordinate,  and at a distance $3(\ln n)^2-1$
such that ``$w_2$ appears in the middle of the segment
$\bar{xy}$''. By this we mean that there exists a canonical vector $\vec{e}$
(such a vector consists of only one non-zero entry which is $1$ or $-1$),
so that
$$w_1w_2w_3=\xi_{x}\xi_{x+\vec{e}}\xi_{x+2\vec{e}}\ldots \xi_{y},$$
where $w_1$ and $w_2$ are words of size $(\ln n)^2$ and
$w_1w_2w_3$ designates the concatenation of $w_1$, $w_2$ and $w_3$.
By the event $B^n_2$, up to time $n^2$ there is at least one particle
which will go from $x$ to $y$ in exactly $3(\ln n)^2-1$ steps. That is
up to time $n^2$ there is a particle which follows the ``straight path
from $x$ to $y$'' given by
$$x,x+\vec{e}, x+2\vec{e},\ldots,y.$$
When doing so, we will see in the observations the string $w_1w_2w_3.$
Thus the triple $(w_1,w_2,w_3)$ satisfies the first criteria
of the first phase of $\Lambda^n$. 
It needs also to pass the second criteria to be selected. 

To see that $(w_1,w_2,w_3)$ satisfies second the criteria, let us assume that  $w$ is a word
of size $(\ln n)^2$ so that the concatenation
$w_1ww_2$ appears in the observation before time $n^4$. Then there exists a nearest
neighbor walk path $R$ of length $3(\ln n)^2-1$ generating $w_1ww_2$ on 
$\xi_{\mathcal{K}(n^4)}$.
By this we mean that $im R\subset \mathcal{K}(n^4)$ and
$\xi\circ R=w_1ww_2.$  By the event $B^n_3$ we have that $R((\ln n)^2-1)$
is in the diamond $T_{w_1}$ associated with $w_1$ and $R(2(\ln n)^2)$ is
in the diamond $T_{w_3}$ associated with $w_3$. So, when we take the restriction
of $R$ to the time interval $[(\ln n)^2-1,2(\ln n)^2]$ we get a nearest neighbor
walk going in $(\ln n)^2$ steps from $T_{w_1}$ to $T_{w_3}$. The only way to do this
is to go in a straight way from the point $x+((\ln n)^2-1)\vec{e}$ to $x+(2(\ln n)^2)\vec{e}$.
(The reason being that the distance between $T_{w_1}$ and $T_{w_3}$ is $(\ln n)^2$
and the only pairs of  points $(x',y')$ so that $x'\in T_{w_1}$ and $y'\in T_{w_3}$ 
and located at that distance $(\ln n)^2$ from each
other,  are $x'=x+((\ln n)^2-1)\vec{e}$ and $y'=x+(2(\ln n)^2)\vec{e}$.)
So, during the time interval $[(\ln n)^2-1,2(\ln n)^2]$ we have that $R$ is walking in a
straight way on the middle part of the segment $\bar{xy}$, that is walking in a straight way
from $x'$ to $y'$. Hence, during that time
$R$ is generating in the observation the word $w_2$. This prove that $w=w_2$. Hence, the triple
$(w_1,w_2,w_3)$ also passes the second criteria of the first phase of our algorithm, which implies that
$$w_2\in SHORTWORDS^n,$$
and hence
$$ W(\xi_{\mathcal{K}(4n)})\subset SHORTWORDS^n .$$

It remains to show that if the triple $(w_1,w_2,w_3)$ gets selected through
the first phase of our algorithm (hence passes the two selection criteria given there),
then indeed $w_2$ is a word of $\xi_{\mathcal{K}(n^2)}$. Now, to pass the first
selection criteria, we have that the concatenation $w_1w_2w_3$ must appear before time $n^2$
in the observations. Since the first particle starts at time $0$ in the origin, by time
$n^2$ all the particles must be still contained in the box $\mathcal{K}(n^2)$.
Hence, there must exist a nearest neighbor path $R$ of length $3(\ln n)^2-1$
which generates $w_1w_2w_3$ on $\xi_{\mathcal{K}(n^2)}$. Hence 
$im R\subset \mathcal{K}(n^2)$ and  $w_1w_2w_3=\xi\circ R.$

Assume that the restriction of $R$ to the time interval $[(\ln n)^2-1,2(\ln n)^2]$
would not be a ``straight walk'' on a line.  Then, by the event $B^n_4$
there would exist a modified nearest neighbor walk $R'$ of length $3(\ln n)^2-1$ for which
the following two conditions are satisfied:
\begin{enumerate}
\item  Restricted to the time interval $[(\ln n)^2-1,2(\ln n)^2]$, we have that
$R'$ generates a string $w$ different from $w_2$ on $\xi$.
\item Outside that time interval, $R'$ generates the same observations $w_1$ and $w_3$ as $R$.\end{enumerate}
Summarizing: the nearest neighbor walk $R'$ generates $w_1ww_3$ on 
$\xi_{\mathcal{K}(n^2)}$,
where $w\neq w_2$. But by the event $B^n_5$, every nearest neighbor-walk of length $3(\ln n)^3-1$
in $\mathcal{K}(n^2)$ gets followed at least once by a particle up to time $n^4$. Hence,
at least one particle follows the path of $R'$ before time $n^4$. Doing so it produces
the string $w_1ww_3$ with $w\neq w_2$ before time $n^4$ in the observations. This
implies however that the triple $(w_1,w_2,w_3)$ fails the second selection criteria
for phase one of our algorithm. This is a contradiction, since we assumed
that $(w_1,w_2,w_3)$ gets selected through
the first phase of $\Lambda^n$ (and hence passes the two selection criteria given there).
This proves by contradiction, that $R$ restricted to the time interval $[(\ln n)^2-1,2(\ln n)^2]$
can only be a ``straight walk''. Hence the sequence generated during
that time, that is $w_2$ can only be a word of the scenery $\xi$ in $\mathcal{K}(n^2)$.
This proves that $w_2$ is in $W(\xi_{\mathcal{K}(n^2)})$ and then
$$SHORTWORDS^n\subset W(\xi_{\mathcal{K}(n^2)}).$$ 

\end{proof}

\subsection{Second phase}
\label{s_2phase}
In this phase  the words of $SHORTWORDS^n$ are assembled
into longer words to construct the set of $LONGWORDS^n$ using the 
assembling  rule:  To puzzle two words together of $SHORTWORDS^n$,
the words  must coincide on at least $(\ln n)^2-1$ consecutive
letters. To get a correct assembling we will need that the short
words could be placed together in a unique way.
\begin{lemma}\label{C1}
Let $C_1^n$ be the event that, for all $x,y\in\mathcal{K}(n^2)$, 
the words 
$\xi_{x}\xi_{x+\vec{e}_i}\ldots 
\xi_{x+((\ln n)^2-1)\vec{e}_i}$
and $\xi_{y}\xi_{y+\vec{e}_j}\ldots 
\xi_{y+((\ln n)^2-1)\vec{e}_j}$
are identical only in the case $x=y$ and $\vec{e}_i=\vec{e}_j$.
Then,
$$
P(C_1^n)\geq 1-\exp[2d\ln(2n^2+1)-((\ln n)^2-1)\ln \kappa].
$$
\end{lemma}
\begin{proof}
Let $x$ and $y$ be two points in $\mathcal{K}(n^2)$ and define the
event
$$C_{x,y}=\{\xi_{x}\xi_{x+\vec{e}_i}\xi_{x+2\vec{e}_i}
\ldots \xi_{x+((\ln
n)^2-1)\vec{e}_i}=\xi_{y}\xi_{y+\vec{e}_j}\xi_{y+2\vec{e}_j}\ldots
\xi_{y+((\ln n)^2-1)\vec{e}_j}\},$$
with $\vec{e}_i$ and $\vec{e}_j$ two canonical vectors in $\Z^d$.
Clearly,
$$C_1^n=\bigcap_{x,y}C_{x,y},$$
where the intersection above is taken over all
$(x,y)\in\mathcal{K}(n^2)$,
 and it leads to
\begin{eqnarray}
P\big(C_1^{nc}\big)\leq\sum_{x,y}P(C_{x,y}^c).
\end{eqnarray}
If $x=y$ and $\vec{e}_i\neq\vec{e}_j$, observe that  the words
intersect themselves only at the first position and since the
scenery is i.i.d., then $P(C_{x,y}^c)=\Big(\frac{1}{\kappa
}\Big)^{(\ln n)^2-1}$. On the other hand, when $x\neq y$, it is
somewhat similar to the previous case, i.e., the words intersect
themselves at most  by only one position, thus
$P(C_{x,y}^c)\leq\Big(\frac{1}{\kappa }\Big)^{(\ln n)^2-1}$.
Hence, 
\begin{eqnarray}
P\big(C_1^{nc})&\leq&(2n^2+1)^{2d}\Big(\frac{1}{\kappa }\Big)^{(\ln
n)^2-1}\nonumber\\
&<&\exp[2d\ln(2n^2+1)-((\ln n)^2-1)\ln \kappa ].
\end{eqnarray}
\end{proof}

\begin{lemma}\label{C}[The second phase works.]
Let $C^n$ designate the event that  every word of size
$4n$  in $\xi_{\mathcal{K}(4n)}$ is  contained in
$LONGWORDS^n$, and
all  the words in  $LONGWORDS^n$  belong to
$\mathcal{W}_{4n}(\xi_{\mathcal{K}(n^2)})$, i.e., 
$$\mathcal{W}_{4n}(\xi_{\mathcal{K}(4n)})\subseteq LONGWORDS^n\subseteq 
\mathcal{W}_{4n}(\xi_{\mathcal{K}(n^2)}),$$
where $\mathcal{W}_{4n}(\xi_{\mathcal{K}(4n)})$ and $\mathcal{W}_{4n}(\xi_{\mathcal{K}(n^2)})$ 
are the set of all  words of size $4n$ in $\xi_{\mathcal{K}(4n)}$
and $\xi_{\mathcal{K}(n^2)}$ respectively, then,
$$ 
B^n\cap C_1^n\subset C^n.
$$
\end{lemma}
\begin{proof} 
Let $W(\xi_{\mathcal{K}(4n)})$ and $W(\xi_{\mathcal{K}(n^2)})$ 
be the set of all  words of size $(\ln n)2$ in
$\xi_{\mathcal{K}(4n)}$ and $\xi_{\mathcal{K}(n^2)}$ respectively.
Once the first phase has worked, i.e., when $B^n$ occurs  we have 
$$
W(\xi_{\mathcal{K}(4n)})\subseteq SHORTWORDS^n\subseteq
 W(\xi_{\mathcal{K}(n^2)}).
$$
If the short words can be placed together in a unique way using 
some  assembling  rule,   and  it is done until getting strings of
total exactly equal to $4n$, then every word of size
$4n$  in $\xi_{\mathcal{K}(4n)}$ is  contained in the set of
assembled 
words, i.e., in
$LONGWORDS^n$. On the other hand, if the assembled process is made
using 
all words in $SHORTWORDS^n$, then all  the words in  $LONGWORDS^n$ 
belong to
$\mathcal{W}_{4n}(\xi_{\mathcal{K}(n^2)})$ because $SHORTWORDS^n
\subseteq W(\xi_{\mathcal{K}(n^2)})$.

Under the assembling rule given in Lemma~\ref{C1} we conclude that
 $B^n\cap C_1^n\subset C^n.$
\end{proof} 

\subsection{Third phase}\label{s_3phase}
In this phase we use the previous two phases of $\varLambda^n$ 
but with the parameter
$n^{0.25}$ instead of $n$. The idea is to take one long word from 
$LONGWORDS^{n^{0.25}}$, say $v$, which will be of size
$4n^{0.25}$ instead of $4n$, then,  select any long word from
 $LONGWORDS^n$, say $w$,
which contains $v$ in its middle. In this manner we should
hopefully  get a word which has its middle not further than $\sqrt{n}$ from the origin.  

In the next lemma we show that when the first two phases of our
algorithm with parameter $n$ as well as $n^{0.25}$ both work, then the third phase must work as well.

\begin{lemma}\label{D}[The third phase works.]
Let $D^n$ be the event that the third phase of $\Lambda^n$ works. 
This means that the long-word of $LONGWORD^n$ selected by the third phase has its 
center
not further than $n^{0.5}$ from the origin.
 Thus $$C^n_1\cap C^n\cap C^{n^{0.25}}\subset
D^n.$$
\end{lemma}

\begin{proof}
Consider any word  from $LONGWORD^n$ that contains  in its middle 
a word from
$LONGWORDS^{n^{0.25}}$, i.e., take  
$$w=\xi_{x}\xi_{x+\vec{e_i}}\xi_{x+2\vec{e_i}}\ldots \xi_{x+4n-1\vec{e_i}}$$
in $LONGWORD^n$ such that 
$$v=\xi_{x+2n\vec{e_i}},\dots,\xi_{x+(2n+4n^{0.25}-1)\vec{e_i}}$$
belongs to  $LONGWORDS^{n^{0.25}}$.  

By $C^{n^{0.25}}$ the first two phases of the algorithm with parameter $n^{0.25}$
work. This implies that $v$ is a word (of size $4n^{0.25}$)
contained in $\xi_{\mathcal{K}(n^{0.5})}$. It is not difficult to see  that  $C^n_1$ implies that
any word of that size which appears in $\xi_{\mathcal{K}(n^2)}$,
appears in a ``unique position'' therein. 
By $C^n$ all the words of $LONGWORD^n$ are contained in $\xi_{\mathcal{K}(n^2)}$.
Thus, when a word $w$ of $LONGWORD^n$ contains a word $v$ of $LONGWORD^{n^{0.25}}$,
then the two words must lie (in $\mathcal{K}(n^2)$) on top of each other in a unique
way, which implies that the middle of $w$ is also the middle of $v$.
By $C^{n^{0.25}}$ we have that the middle of $v$ (the way $v$ appears in $\xi_{\mathcal{K}(n^{0.5})}$) is not further from the origin than $n^{0.5}$. Hence, the middle of $w$ (in where it appears in $\xi_{\mathcal{K}(n^2)}$)  is also not further than $n^{0.5}$ from the origin. This finishes
our proof.
\end{proof}


\subsection{Fourth phase}
\label{s_4phase}
For the fourth and last phase to work correctly,
we need to be able to identify which words
contained in $\xi_{\mathcal{K}(n^2)}$ are neighbors of each other.
Let us give the definition of neighboring words. 

Let $I$ be a box of $\mathbb{Z}^d$ and $w$ and $v$ be two words
(of the same length) contained in $\xi_{I}$. We say that $w$ and $v$ are {\it neighbors
of each other} if there exist $x\in I$ and  $\vec{u}, \vec{s}\in\{\pm\vec{e}_i, i=1,\ldots,d\}$
such that:
$$
w=\xi_{x}\xi_{x+\vec{u}}\xi_{x+2\vec{u}}\ldots \xi_{k\vec{u}}
$$
and 
$$
v=\xi_{x+\vec{s}}\xi_{x+\vec{u}+\vec{s}}\xi_{x+2\vec{u}+\vec{s}}
\ldots 
\xi_{x+k\vec{u}+\vec{s}},
$$
where $\vec{s}$ is orthogonal to $\vec{u}$ and all the points
$x+\vec{s}+i\vec{u}$ and $x+i\vec{u}$ are in $I$ for all
$i=0,1,2,\ldots,k$.
 
In other words, two words $w$ and $v$ contained in the restriction
$\xi_{I}$ are called neighbors if we can read them in positions
which are parallel to each other and at distance~$1$.


\begin{lemma}\label{4phaseworks}[The fourth phase works.]
Let $F^n$ denote the event that for the words of
$LONGWORDS^n$  the three conditions in the forth phase of $\Lambda^n$
allows to correctly identify and chose neighbors.
Interestingly, the events  $B^n_2$, $B^n_3$ and $C^n_1$ are
enough for  $F^n$ to occur. That is,
$$ B^n_2\cap B^n_3 \cap C^n_1\subset F^n$$
\end{lemma}

\begin{proof} We need to show that when all the events $B^n_2$, $B^n_3$ and  $C^n_1$ occur,
then we identify  correctly  and  chose the long words from
$LONGWORDS^n$ which are neighbors of each  other. 

If two words $v$ and $w$ belonging to $LONGWORDS^N$ were selected by the fourth phase of our algorithm to be put on top of each other, this means that $\Lambda^n$ ``estimated'' that $v$ and $w$ were neighbors, is because  the following three conditions  were satisfied:
\begin{enumerate}
\item{} There   exist  $4$ words $v_a$, $v_b$, $v_c$ and
$w_b$ having all size $(\ln n)^2$ except for $v_b$ which has
size $(\ln n)^2-2$ and such that, the concatenation $v_av_bv_c$ is contained in $v$, whilst up to
time $n^4$ we observe at least once $v_aw_bv_c$.
\item{} The word $w_b$ is contained in $w$.
\item{} The  position of the middle letter of 
$v_b$ in $v$ should be the same as the middle letter  
position of $w_b$ in $w$.
{\footnotesize  }
\end{enumerate}
Assume that the three previous conditions are satisfied, and let $\vec{u}$ be the  direction of the word $v$, so, we have two points $x$ and $y$ in 
$\mathbb{Z}^d$ such that $y=x+((\ln n)^2-1) \vec{u}$ and
\begin{itemize}
\item  to the left (with respect to the  direction of $\vec{u}$) from $x$ we read $v_a$:
$$v_a=\xi_{x-\vec{u}((\ln n)^2+1)}\xi_{x-\vec{u}((\ln n)^2+2)}
\xi_{x-\vec{u}((\ln n)^2+3)}\ldots \xi_x,$$
\item between $x$ and $y$ we read $v_b$:
$$
v_b=\xi_{x+\vec{u}}\xi_{x+2\vec{u}}\ldots \xi_{y-\vec{u}}, and
$$
\item to the right of $y$ we read $v_c$:
$$
v_c=\xi_{y}\xi_{y+\vec{u}}\xi_{y+2\vec{u}}\ldots 
\xi_{y+((\ln n)^2-1)\vec{u}}.
$$
\end{itemize}
By the first condition we know that   up to time $n^4$ we observe at least once the concatenated word $v_aw_bv_c$. Hence, there exists a nearest neighbor walk  $R$ on the time interval $[1,3(\ln n)^2]$ which generated $v_aw_bv_c$ on $\xi_{\mathcal{K}(n^4)}$, so that
$$
\xi\circ R=v_aw_bv_c.
$$
Note that  after time $2(\ln n)^2$ $R$  generates~$v_c$,  then by $B^n_3$ we have that
$R(2(\ln n)^2+1)$ must be in the diamond $T_c$ associated with $v_c$. 
Similarly, we get that $R((\ln n)^2)$ must be in the diamond
$T_a$ associated with $v_a$. 
%

Now observe that  to go in $(\ln n)^2+1$ steps with a nearest neighbor walk 
from the diamond $T_a$ to the diamond $T_c$,
there are only 3 possibilities (remember that  $x$ and $y$
are at distance $(\ln n)^2-1$ from each other):
\begin{itemize}
\item[I)] Going from $x$ to $y$ always making steps with respect to $\vec{u}$ (in 2 dimensions, say to the right), and once making one step with respect to $-\vec{u}$ (one step to the left). No steps in the directions orthogonal to~$\vec{u}$.
\item[II)] Starting at $x$, make one step in some
direction orthogonal to~$\vec{u}$, say with respect to $\vec{s}$ (in 2 dimensions,  say one step up) then all steps with respect to $\vec{u}$ (to the right), and once making one step with respect to $-\vec{s}$ (one step down) in order to reach $y$. 
\item[III)] Starting at $x+\vec{s}-\vec{u}$ instead of $x$ and arriving in $y+\vec{s}+\vec{u}$ instead of $y$.
\end{itemize}

Since the nearest neighbor walk $R$ between time $(\ln n)^2 $ and
$2(\ln n)^2+1$ must be walking from the diamond $T_a$ 
to the diamond  $T_c$ in $(\ln n)^2+1$ steps, then it  must satisfy during that time one of the three conditions above.

If condition II holds, then $w_b$ is the word which is ``written''
in the scenery $\xi$ between $x+\vec{s}$ and $y+\vec{s}$, i.e., 
$$
w_b=\xi_{x+\vec{s}}\xi_{x+\vec{s}+\vec{u}}\xi_{x+\vec{s}+2\vec{u}}
\ldots 
\xi_{y+\vec{s}}.
$$ 
That shows that the line between the two points $x+\vec{s}$ and
$y+\vec{s}$ is where the word~$w_b$ is written in the scenery. 
Now observe that  $w$ appears in $\xi_{\mathcal{K}(n^2)}$ because $w$ is in $LONGWORDS^N$, and by  the second condition  $w$ must contain the word $w_b$, so $w$  must contain the points $x+\vec{s}$ and $y+\vec{s}$.  By the event $ C^n_1$, any word of size $(\ln n)^2$ appears only in one place in $\xi_{\mathcal{K}(n^2)}$, then as $w_b$ has size  $(\ln n)^2$, the place where the word $w$ is written in
$\xi_{\mathcal{K}(n^2)}$ is a line parallel to the line $\bar{xy}$ and at distance~$1$.
This means that the word $w$ is a neighbor of the word $v$ in  $\xi_{\mathcal{K}(n^2)}$.

When condition III) holds, a very similar argument leads to the  same conclusion
that   $w$ and $v$ are neighbors in $\xi_{\mathcal{K}(n^2)}$.

Finally, when condition I) holds, then we would have that
$v_aw_bv_c$ appears in $\xi_{\mathcal{K}(n^2)}$,  but we have also
that $v_av_bv_c$ appears in $\xi_{\mathcal{K}(n^2)}$ (since we take
the words $w$ and $v$ from the set $LONGWORDS^n$ and since by the event
$C^n_1$ these are words from the scenery $\xi_{\mathcal{K}(n^2)}$).
However the word $w_b$ is longer than $v_b$. This implies that either
$v_a$ or $v_c$ (or both together) must appear 
in two different places in $\xi_{\mathcal{K}(n^2)}$. But this
would contradict the event $C^n_1$. Hence, we can exclude this case.

Thus, we have  proven that if the events  $B^n_3$ and $C^n_1$ all hold,  then the words selected to be put next to each other in phase 4 of the algorithm are really 
neighbors in the way they appear in the restricted scenery $\xi_{\mathcal{K}(n^2)}$. 
So, when, we use the criteria of the 4th phase of the algorithm to determine the words
which should be neighbors in the scenery we make no mistake, by
identifying as neighbors whilst they are not. 

The next question is whether for all the neighboring words $v$ and $w$ in
$\xi_{\mathcal{K}(4n)}$ we recognize them as neighbors, when we apply the 
fourth phase of our algorithm. This is indeed true, due to
the event $B^n_2$. Let us explain this in more detail.
Let $v$ and $w$ be to neighboring words in
$\xi_{\mathcal{K}(4n)}$. We also assume that both $v$ and $w$ have length $4n$, are written in the direction of~$\vec{u}$, 
and $w$ is parallel and at distance 1 (in the direction
$\vec{s}$ perpendicular to $\vec{u}$) of $v$. 
Now, take anywhere
approximately in the middle of $v$, three consecutive strings
$v_a$, $v_b$ and $v_c$. Take them so that $v_a$ and $v_c$ have size $(\ln
n)^2$, but $v_b$ has size $(\ln n)^2-2$. Hence, we assume that the
concatenation $v_av_bv_c$ appears somewhere in the middle of $v$. Let
$x\in\mathbb{Z}^d$ 
designate the point where $v_a$ ends and let $y$ be the point where $v_c$ starts.
Hence
$$v_c=\xi_{y}\xi_{y+\vec{u}}\xi_{y+2\vec{u}}\ldots \xi_{y+((\ln
n)^2-1)\vec{u}}$$
Also the points $x$ and $y$ are 
at distance $(\ln n)^2-1$ from each other and on the line directed along~$\vec{u}$. 
Finally,
$$
v_a=    \xi_{x-((\ln n)^2+1)\vec{u}}
\ldots\xi_{x-2\vec{u}}\xi_{x-\vec{u}}\xi_{x}.
$$
So, the word $v$ is written on the line passing  through $x$ and $y$.
Note that the word $w$ is written on the line which passes  through the points
$x+\vec{s}$ and $y+\vec{s}$. 
Now, because of the event $B^n_2$, for
every nearest neighbor walk path of length $3(\ln n)^2-1$ contained in the region
$\mathcal{K}(4n)$ there is at least one particle following that path up to time~$n^2$.

Take the nearest neighbor walk path 
$R$ on the time interval $[1,3(\ln n)^2]$ which starts in 
$x-((\ln n)^2-1)\vec{u}$ and then goes in the direction of~$\vec{u}$
$(\ln n)^2$ steps. Next  goes one step in the direction of~$\vec{s}$ (and hence
reaches the line where $w$ is written), and  then, $R$ takes $(\ln n)^2$ steps in the
direction of~$\vec{u}$
and reads part of the word $w$. That part we designate by $w_b$.
From there one step in the direction~$(-\vec{s})$, to reach the 
point~$y$ and then all remaining
$(\ln n)^2$ steps are taken in the
direction~$\vec{u}$. During those remaining steps
the walk generates the color record $v_c$. Such a nearest neighbor
walk generates thus the color record $v_aw_bv_c$:
$$\xi\circ R=v_aw_bv_c$$

Hence, by $B^n_2$ at least one particle up to time $n^2$ follows $R$ and 
generates the color record $v_aw_bv_c$. Similarly we can chose a neighbor path that generates $v_av_bv_c$. Since $w_b$ and $v_b$ are contained in $w$ and $v$, respectively, then $w$ and $v$ get selected as neighbors by the forth step of our algorithm.
\end{proof}

\begin{lemma}\label{works}[The algorithm $\Lambda^n$ works.]
Let $A^n$ be the event that  $\Lambda^n$ works. This means that  the outcome after the fourth
phase of $\Lambda^n$ is ``correct''. Thus
$$B^n\cap C^n\cap \ D^n\cap F^n \subset A^n$$
\end{lemma}

\begin{proof} 
Recall  we said our algorithm works as a whole correctly, 
if there exists a box $I$ with size $4n$ with center closer than $n^{0.5}$ from 
the origin and such that the restriction $\xi_{I}$ is equivalent to our
reconstructed piece of
scenery. If the last phase of the algorithm works correctly, then
we would like to see that outcome. The event that 
the outcome after the fourth
phase is ``correct'', is denoted by $A^n$ (as already mentioned).
Now, assume that we have the correct short and  long words constructed in
phase one and two, that is, assume that the events $B^n$ and  $C^n$ occur. Assume also that the third phase
of our algorithm works correctly and we get the one long word
picked at the end of phase 3 to be close enough to the origin, hence the event
$D^n$ occurs. When these three events occur (phase 1, 2  and 3 work)  for the final phase 
of the algorithm to work, we then only need to identify which words
of $LONGWORDS^n$ are neighbors of each
other (remember that the words in   $LONGWORDS^n$ are words of  $\xi_{\mathcal{K}(n^2)}$),
so, what we really need  after is finding a way to determine which words
of $\xi_{\mathcal{K}(n^2)}$ are neighbors of each other, and that holds by $F^n$.

Already at this point we observe that, if we have a collection
of objects to be placed in a box of $\Z^{d}$ (in a way that all 
the sites of this box become occupied by exactly one object), and
we know which objects should be placed in neighboring sites,
 then there is a unique (up to translations/reflections)
way to do it. This will assure that the reconstruction works
correctly once we identified the neighboring words.
\end{proof}

\begin{proof}\label{proofT22}[Proof of Theorem \ref{maintheorem}.]  
The last lemma above, Lemma \ref{works}, tells us that for the algorithm
$\Lambda^n$ to work correctly, we just need the events 
$$B^n_1,B^n_2,B^n_3,B^n_4,B^n_5,C^n_1,C^{n^{0.25}}$$
to all hold simultaneously.
Thus, Theorem~\ref{n-theorem}  follows from Lemmas~\ref{B1},
\ref{B2}, \ref{B3}, \ref{B4}, \ref{B5}, \ref{C1} and  \ref{C}.
\end{proof}


%

\section{The infinite scenery: Proof of Theorem \ref{teorema1}}

How do we now reconstruct the infinite scenery $\xi$? So far we
have seen
only methods to reconstruct a piece of $\xi$ on a box of size
$4n$
close to the origin. The algorithm was denoted by $\Lambda^n$,
and the event that it works correctly is designated by the event
$A^n$.
By working correctly, we mean that the piece reconstructed in
reality
is centered not further than $\sqrt{n}$ from the origin. In general, it 
is not possible to be sure about the exact location. So, instead we are
only
able to figure out the location up to a translation of order
$\sqrt{n}$. Also, the piece is reconstructed
only up to equivalence, which means that we do not know in which
direction it
is rotated in the scenery or if it was reflected.
Now, the probability that the reconstruction algorithm
$\Lambda^n$
at level $n$ does not work is small in the sense that it is
finitely summable
over $n$:
$$\sum_nP(A^{nc})<\infty.$$
So, by Borel-Cantelli lemma, when we apply all the reconstruction
algorithms
$\Lambda^1, \Lambda^2, \ldots$, we are almost sure that only
finite many
of them will not work. We use this to assemble the sceneries 
and get the whole scenery a.s. Let us call $\xi^n$ the piece
reconstructed
by $\Lambda^n$. Hence,
$$\xi^n:=\Lambda^n(\chi_0\chi_1\ldots\chi_{n^4}),$$
where $\xi^n$ is a piece of scenery on a box $\mathcal{K}(4n)$.
The next task is to put the reconstructed pieces~$\xi^n$ 
together so that their relative
position to each other is the same  in the scenery~$\xi$.
For this we will use the following rule. 
We proceed in an inductive way in~$n$:
\begin{enumerate}
\item{} Let $\bar{\xi}^n$ designate the piece $\xi^n$ which has
been moved so
as to fit the previously placed pieces. Hence, $\bar{\xi}^n$ is
equivalent to~$\xi^n$.
\item{} We place $\xi^n$ by making it coincide with the previously
placed $\xi^{n-1}$ on a box of side length at least $\sqrt{n}$. In
other words, $\bar{\xi}^n$ is defined to be any piece of the scenery equivalent to
$\xi^n$, and such that on a restriction to a box of size $\sqrt{n}$
it coincides with $\bar{\xi}^{n-1}$. If no such box of size
$\sqrt{n}$
can be found in $\bar{\xi}^{n-1}$ which is equivalent to
a piece of the same size in $\xi^n$, then we ``forget'' about the
previously
placed pieces and just put the piece $\xi^n$ on the box
$\mathcal{K}(4n)$,
that is, we center it around the origin.
\item{}The end result after infinite time is our reconstructed
scenery
denoted by 
$$\bar{\xi}:\mathbb{Z}^d\rightarrow\{0,1,\ldots,\kappa\}.$$
We will prove that a.s.\ $\bar{\xi}$ is equivalent to $\xi$.
So, $\bar{\xi}$ represents our reconstructed $\xi$ (since
reconstruction
in general is only possible up to equivalence). For $\bar{\xi}$
we simply take the point-wise limit of the $\bar{\xi}^n$:
$$\bar{\xi}(\vec{z}):=\lim_{n\rightarrow\infty}\tilde{\xi}^n.$$
For the above definition to be meaningful, take $\tilde{\xi}^n$ to
be the extension of $\bar{\xi}^n$ to all $\mathbb{Z}^d$ by adding
$0$'s 
where $\bar{\xi}^n$ is not defined.
\end{enumerate}

To conclude the proof of Theorem~\ref{teorema1}, it is enough to prove
that the above algorithm defines a scenery $\bar{\xi}$ which is almost
surely equivalent to $\xi$:
\begin{theorem} We have that
$$P(\xi\approx \bar{\xi})=1.$$
\end{theorem}

\begin{proof} Let $G^n$ be the event that in the restriction 
of $\xi$ to the box $\mathcal{K}(2n+2)$ any two restrictions
to a box of size $\sqrt{n}$ are different of each other.
By this we mean, that if $V_1$ and~$V_2$ are two boxes
of size $\sqrt{n}$ contained in $\mathcal{K}(2n+1)$,
then if the restriction $\xi_{V_1}$ is equivalent to $\xi_{V_2}$
then $V_1=V_2$. Also, for $G^n$ to hold, we require that
for any box $V_1\subset \mathcal{K}(2n+2)$ of size $\sqrt{n}$ 
the only reflection and/or
rotation which leaves $\xi_{V_1}$ unchanged is the identity.

Now note that when the event $G^n$ occurs, and the piece $\xi^{n+1}$
and $\xi^n$ are correctly reconstructed in the sense defined before,
then our placing them together works properly. This means that in
that case, the relative position of $\bar{\xi}^{n+1}$ and
$\bar{\xi}^n$ is that same as the corresponding pieces in $\xi$.
It is elementary to obtain that the event $G^n$ has  
probability at least $1-c_1n^{d}e^{-c_2n^{d/2}}$.
This means that $\sum_nP(G^{nc})<\infty$, and so,
 by Borel-Cantelli lemma, all but a finite number of the events
$G^n$ occur. Also, we have seen that the algorithm at level~$n$
has high probability to do the reconstruction correctly, and 
$\sum_nP(A^{nc})<\infty.$

Hence, again by Borel-Cantelli lemma,
all but a finite number of the reconstructed scenery~$\xi^n$
will be equivalent to a restriction of $\xi$ to a box close
to the origin. We also see that the close to the origin
for the box of the $n$-th reconstruction means not
further than~$\sqrt{n}$.  Thereof we have
that all but a finite number of the pieces $\bar{\xi}^n$ are positioned correctly
with respect to each other. Since we take a limit for getting
$\xi$, a finite number of $\bar{\xi}^n$'s alone have no effect on
the limit, and so the algorithm works.
\end{proof}


\bibliographystyle{plain}

\begin{thebibliography}{19}  

\bibitem{denHollanderSteif}
\textsc{den Hollander, F.} and \textsc{Steif, J.E.} (1997)
Mixing properties of the generalized ${T},{T}\sp {-1}$-process.
\textit{J. Anal. Math.} \textbf{72} 165--202.

\bibitem{BenjaminiKesten}  
\textsc{Benjamini, I.} and \textsc{Kesten, H.} (1996)
Distinguishing sceneries by observing the scenery along a random walk path.
\textit{J. Anal. Math.} \textbf{69} 97--135.

\bibitem{Howard1}  
\textsc{Howard, C.D.} (1996) 
Detecting defects in periodic scenery by random walks on $\Z$.
\textit{Random Structures Algorithms.}  \textbf{8} (1), 59--74.

\bibitem{Howard2}  
\textsc{Howard, C.D.} (1996)
Orthogonality of measures induced by random walks with scenery.
\textit{Combin. Probab. Comput.}  \textbf{5}  (3), 247--256.

\bibitem{Howard3}  
\textsc{Howard, C.D.} (1997) 
Distinguishing certain random sceneries on $\Z$ via random walks. 
\textit{Statis. Probab. Lett.} \textbf{34} (2), 123--132.

\bibitem{KestenSingDef}  
\textsc{Kesten, H.} (1996) 
Detecting a single defect in a scenery by observing the scenery along a random walk path. 
\textit{Ito's Stochastic Calculus and Probability theory. Springer, Tokyo.}  171--183.

\bibitem{KestenReview} 
\textsc{Kesten, H.} (1998) 
Distinguishing and reconstructing sceneries from  observations along random walk paths.
\textit{Microsurveys in Discrete Probability: DIMACS.  Amer. Math. Soc., Providence, RI.} \textbf{41}, 75--83. 

\bibitem{Kalikow}  
\textsc{Kalikov, S.A.} (1982)
T,T $^{-1}$ transformation is not loosely Bernoulli.
\textit{Annals of Math.}  \textbf{115}, 393-409.


\bibitem{messagetext}
\textsc{Lember, J.} and \textsc{Matzinger, H.} (2008)
Information recovery from a randomly mixed up message-text.
\textit{Electronic Journal of Probability.} \textbf {13}, 396--466.

\bibitem{Lindenstrauss}
\textsc{Lindenstrauss, E} (1999)
Indistinguishable sceneries.
\textit {Random Structures Algorithms.} \textbf{14},  71--86.

 \bibitem{Lowe-Matzinger-two-dim}
\textsc{ L\"owe, M} and \textsc{Matzinger, M.} (2002)
Scenery reconstruction in two dimensions with many colors.
\textit {Ann. Appl. Probab.} \textbf{12} (4), 1322--1347.

\bibitem{Lowe-Matzinger-Merkl2001}
\textsc{ L\"owe, M.  Matzinger, H.}  and \textsc{Merkl, F.}  (2004)
Reconstructing a multicolored random scenery seen along a random walk
path with bounded jumps.
\textit {Electronic Journal of Probability.} \textbf{15}, 436--507.

\bibitem{Heini2color} 
Reconstructing a 2-color scenery by  observing it along a simple random walk path.
\textit {Ann.Appl.Probab.} \textbf{15}, 778--819. 

\bibitem{Heini3color} 
\textsc{Matzinger, H.} (1999)
Reconstructing a three-color scenery by observing it along 
a simple random walk path. 
\textit{Random Structures Algorithms.} \textbf{15} (2), 196--207.

\bibitem{HeinisDiss} 
\textsc{Matzinger, H.} (1999)
Reconstructing a 2-color scenery by observing it  along a simple random walk path with holding.
\textit {Cornell University}.

\bibitem{Matzinger-Angelica}
\textsc{Matzinger, H.} and \textsc{Pachon, A.} (2011)
DNA approach to scenery reconstruction.
\textit{Stochastic Process. Appl.} \textbf{121} (11), 2455--2473.


\bibitem{popovscen}
\textsc{Matzinger, H.} and \textsc{Popov, S.} (2008)
Detecting a local perturbation in a continuous scenery.
\textit {Electronic Journal of Probability.} \textbf{12}, 1103--1120.

\bibitem{Matzinger-Rolles2001} 
\textsc{Matzinger, H.} and \textsc{Rolles, S.W.W.} (2003)
Reconstructing a random scenery observed with random 
errors along a random walk path.
\textit{Probab. Theory Related Fields.} \textbf{125} (4), 539--577.


\bibitem{Matzinger-Rolles-small} 
\textsc{Matzinger, H.} and \textsc{Rolles, S.W.W.} (2006)
Finding blocks and other patterns in a random coloring of  $\Z$.
\textit{Random Structures Algorithms.} \textbf{28}, 37--75.

\bibitem{Matzinger-Rolles-polynomial} 
\textsc{Matzinger, H.} and \textsc{Rolles, S.W.W.} (2003)
Reconstructing a piece of scenery with polynomially many  observations.
\textit {Stochastic Process.Appl.} \textbf{107} (2), 289--300.

\bibitem{Ornstein}
\textsc{Ornstein, D.} (1971)
A Kolmogorov automorphism that is not a Bernoulli shift.
\textit{Matematika.} \textbf{15} (1), 131--150 (In Russian).

\bibitem{Popov-Angelica}
\textsc{Popov, S.} and \textsc{Pachon, A.} (2011)
Scenery reconstruction with branching random walk.
\textit {Stochastics.} \textbf{83} (2), 107--116.

\bibitem{Weiss}
\textsc{Weiss, B.} (1972)
The isomorphism problem in ergodic theory.
\textit {Bull. Amer. Math. Soc.} \textbf{78},
668--684. MR304616.

\end{thebibliography}

\end{document}